\documentclass[12pt,a4paper]{amsart}

\makeatletter
\usepackage[english]{babel}
\usepackage[utf8]{inputenc}
\usepackage[square,sort,comma,numbers]{natbib}
\usepackage{graphicx}
\usepackage{amsfonts}
\usepackage{amsthm}
\usepackage{amsmath,amssymb}
\usepackage{xparse}
\usepackage{tikz-cd}
\usepackage{verbatim}
\usepackage{mathtools}
\usepackage[new]{old-arrows}
\setcitestyle{square}

\numberwithin{equation}{section}

\theoremstyle{plain}
	\newtheorem{thm}[equation]{Theorem}
	\newtheorem{prop}[equation]{Proposition}
	\newtheorem{lem}[equation]{Lemma}

	\newtheorem{cor}[equation]{Corollary}

	\newtheorem{lem/defn}[equation]{Lemma/Definition}

\theoremstyle{definition}
	\newtheorem{defn}[equation]{Definition}

\theoremstyle{remark}
	\newtheorem{rem}[equation]{Remark}

\def\nc{\newcommand}
\def\on{\operatorname}
\def\co{\colon\thinspace}

\newcommand{\RomanNumeralCaps}[1]
    {\MakeUppercase{\romannumeral #1}}

\nc{\edit}[1]{\marginpar{\footnotesize{#1}}}
\newcommand{\lv}{\lvert}
\newcommand{\rv}{\rvert}
\nc{\C}{\mathcal{C}}
\nc{\Z}{\mathbb{Z}}
\nc{\z}{\mathbb{Z}}
\nc{\PP}{\mathbb{P}}
\nc{\R}{\mathbb{R}}
\nc{\f}{\mathbb{F}}
\nc{\fp}{\mathbb{F}_p}
\nc{\sph}{\mathbb{S}}
\nc{\hfp}{H\mathbb{F}_p}
\nc{\T}{\mathbb{T}}
\nc{\wdg}{\wedge}
\nc{\wdgfp}{\wedge_{H\mathbb{F}_p}}
\newcommand {\we} {\ensuremath{\tilde\rightarrow}}
\nc{\AAA}{\mathbb{A}}
\nc{\LL}{\mathbb{L}}
\nc{\OO}{\mathcal{O}}

\nc{\X}{\EuScript{X}}
\nc{\sZ}{\EuScript{Z}}

\nc{\id}{{\on{id}}}
\nc\Hom{{\on{Hom}}}
\nc\cone{{\on{cone}}}
\nc{\Rep}{{\on{Rep}}}
\nc\Ob{{\on{Ob}}}
\nc\Spec{{\on{Spec}}}
\nc\Mod{{\on{Mod}}}
\nc\coMod{{\on{coMod}}}
\nc\Perf{{\on{Perf}}}
\nc\End{{\on{End}}}
\nc{\into}{\hookrightarrow}
\nc{\tr}{\on{tr}}
\nc{\ev}{\on{ev}}
\nc{\im}{\on{im}}
\nc{\Mot}{\on{Mot}}
\nc{\pt}{\on{pt}}
\nc{\coker}{\on{coker}}
\nc{\rk}{\on{rank}}
\nc{\TOP}{\on{Top}_{\mathbb{C}}^{s}}
\nc{\gr}{\on{gr}}
\nc{\Catperf}{\text{Cat}^{\text{perf}}}
\nc{\Sym}{\on{Sym}}
\nc{\xra}{\xrightarrow}
\nc{\lra}{\xleftarrow}
\nc{\Bet}{\mathbf{Betti}_{X}}
\nc{\codim}{\on{codim}}
\nc{\Fred}{\on{Fred}}
\nc{\colim}{\on{colim}}
\nc{\KK}{{\bf K}}
\nc{\onto}{\twoheadrightarrow}
\nc{\A}{\mathbb{A}}
\nc{\Aff}{\on{Aff}}
\nc{\SH}{\on{SH}}
\nc{\QCoh}{\on{QCoh}}
\nc{\Alg}{\on{Alg}}
\nc{\Br}{\on{Br}}
\nc{\ta}{\widetilde{\a}}
\nc{\Shv}{\on{Shv}}
\nc{\GG}{\mathbb{G}}
\nc{\red}{\color{red}}
\nc{\blue}{\color{blue}}
\nc{\an}{\on{an}}
\nc{\D}{\on{D}}
\nc{\Pre}{\on{Pre}}
\nc{\assact}{\on{Ass}_{\on{act}}^{\otimes }}
\nc{\spact}{\on{Sp}_{\on{act}}^{\otimes }}
\nc{\spactprod}{{\on{Sp}^{\Z}_{\on{act}}}^{\otimes }}
\nc{\qc}{\on{qc}}
\nc{\op}{\on{op}}
\nc{\shEnd}{{\mathcal End}}
\nc{\Top}{\on{Top}}
\nc{\Map}{\on{Map}}
\nc{\Vect}{\on{Vect}}
\nc{\holim}{\on{holim}}

\DeclareMathOperator{\thh}{\ensuremath{\textup{THH}}}

\DeclareMathOperator{\tor}{\ensuremath{\textup{tor}}}

\DeclareMathOperator{\tp}{\ensuremath{\textup{TP}}}
\DeclareMathOperator{\tc}{\ensuremath{\textup{TC}}}
\DeclareMathOperator{\kth}{\ensuremath{\textup{K}}}

\DeclareMathOperator{\hh}{\ensuremath{\textup{HH}}}

\def\A{\mathcal{A}}

\def\a{\alpha}

\def\Perf{\on{Perf}}

\nc{\W}{\mathbb{W}}

\def\QCoh{\on{QCoh}}

\title{Algebraic $K$-theory of $\thh(\fp)$}
\author{Haldun Özgür Bayındır, Tasos Moulinos}
\usepackage{natbib}
\usepackage{graphicx}

\begin{document}
\maketitle
\begin{abstract}

In this work we study the $E_{\infty}$-ring $\thh(\fp)$ as a graded spectrum. Following an identification at the level of $E_2$-algebras with $\fp[\Omega S^3]$, the
group ring of the $E_1$-group $\Omega S^3$ over $\fp$, we show that the grading on $\thh(\fp)$ arises from a decomposition on the cyclic bar construction of the pointed monoid $\Omega S^3$. This allows us to use trace methods to compute the algebraic $K$-theory of $\thh(\fp)$. We also show that as an $E_2$ $H\fp$-ring, $\thh(\fp)$ is uniquely determined by its homotopy groups. These results hold in fact for $\thh(k)$, where $k$ is any perfect field of characteristic $p$. Along the way we expand on some of the methods used by Hesselholt-Madsen and later by Speirs to develop certain tools to study the THH of graded ring spectra  and the algebraic $K$-theory of  formal DGAs. 
\end{abstract}

\section{Introduction}
Trace methods form a central strategy in computing algebraic $K$-theory. There exist canonical \emph{trace maps} between the algebraic $K$-theory and the more accessible Hochschild invariants. In many cases of interest, these invariants faithfully reflect the behavior of algebraic $K$-theory itself. Thus, a refined understanding of such Hochschild  invariants along with these trace maps has made computations in algebraic $K$-theory considerably more feasible. 
This line of attack has been further streamlined in recent years, due in large part to a rich interplay of higher categorical structures and classical homotopy theoretic calculations. 

A central ingredient to the story, at least in the characteristic $p$ setting, is the computation due to B\"okstedt of the homotopy groups of the topological Hochschild homology of $\fp$ as a graded polynomial algebra  on
a generator in degree 2; this has come to be known as B\"okstedt periodicity. As the term may suggest, whenever $R$ is linear over $\fp$, this allows for the periodicity operator in Tate cohomology
$$
t \in \hat{H}^{2}( B \T, \pi_* \thh(R)))
$$
to survive to the $E^\infty$-page of the spectral sequence computing the Tate fixed point spectrum $\thh(R)^{t \T }$.  Working in the modern formulation of cyclotomic spectra due to Nikolaus-Scholze, this often allows for a streamlined computation of the topological cyclic homology  $\tc(R)$.

In this paper we turn this phenomenon of B\"okstedt periodicity on its head by using it to compute the topological cyclic homology of $\thh(k)$ itself, for $k$ any perfect field of characteristic $p$. Before stating our main result, we recall the following celebrated theorem of 
Dundas-Goodwillie-McCarthy:

\begin{thm}[Dundas-Goodwillie-McCarthy] \label{thm dundasgoodwilliemccarthy}
Let $A\to B$  be a map of connective $\sph$-algebras such that $\pi_0A \to \pi_0B$ is surjective and has nilpotent kernel. Then the square
\begin{equation*}
    \begin{tikzcd}
    K(A) \ar[r,"tr"] \ar[d]& TC(A)\ar[d]\\
    K(B) \ar[r,"tr"] & TC(B)
    \end{tikzcd}
\end{equation*}
is a homotopy pullback square.
\end{thm}
With this result at our disposal, the main theorem in this paper takes on the following form. 

\begin{thm}\label{thm algebraic k theory of thh k}
For every perfect field of positive characteristic $k$, there are isomorphisms
\[\kth_{2r+1}(\thh(k),k) \cong \mathbb{W}_r(k) \]
where $\mathbb{W}_r(k)$ denotes the big Witt vectors of length $r$ in $k$. The   relative algebraic $K$-theory groups are trivial in even degrees. 
\end{thm}

Our starting point for this calculation is the observation from \cite{krause2019b} that $\thh(\fp)$ is equivalent as an $E_1$-algebra to $\fp [\Omega S^3]$, the algebra of chains on the free $E_1$-group  $ \Omega \Sigma  S^2$ on $S^2$, viewed as a pointed space. As topological Hochshild homology is symmetric monoidal, one then obtains an equivalence of $\T$-equivariant spectra 
$$
\thh(\thh(\fp)) \simeq \thh(\fp [\Omega S^3]) \simeq \thh(\fp) \wedge \thh(\sph[\Omega S^3]) 
$$

We improve upon this with the following result, which serves as a key input in the ensuing calculation:   

\begin{thm} \label{thm thhfp as an e2 algebra}
For every perfect field $k$ of positive characteristic, there is an equivalence of $E_2$ $Hk$-algebras
\[\thh(k) \simeq Hk \wdg \sph[\Omega S^3].\]
\end{thm}
Indeed, we prove a stronger statement. We show that there is a unique $E_2$ $k$-DGA with homology ring $k[x]$ for every positive and even $\lv x \rv$, see Theorem \ref{thm e2 dgas with polynomial homology}. We prove this result using the $E_n$ Postnikov extension theory of Basterra and Mandell \cite{basterra2013BP}. 
\begin{cor}\label{cor e1 equivalence of k theories}
There is an equivalence of $E_1$-algebras
\[\kth(\thh(k)) \simeq \kth(Hk \wdg \sph[\Omega S^3]).\]
\end{cor}
Among other things, the splitting in Theorem \ref{thm thhfp as an e2 algebra} provides the multiplicative structure on the (second) iterated topological Hochschild homology groups of $k$. Namely, we obtain a ring isomorphism
\begin{equation} \label{eq thhthhfp ring structure}
    \begin{split}
        \thh_*(\thh(k)) \cong& \pi_*(\thh(k) \wdg \thh(\sph[\Omega S^3]))
        \\
        \cong& \thh_*(k)\otimes \hh^k_*(Hk \wdg \sph[\Omega S^3]) \\
        \cong & k[x_2] \otimes k[y_2]\otimes \Lambda(z_3) \\
    \end{split}
\end{equation}
where subscripts denote the internal degrees and the last factor denotes the exterior algebra over $k$ on a single generator. The third isomorphism is described in \eqref{eq equation above 2}. 
Note the isomorphism in \eqref{eq thhthhfp ring structure} for odd $p$ and $k=\fp$ is due to Veen  \cite{veen2018detecting}. Indeed, Veen calculates the $n$-iterated topological Hochschild homology groups of $\fp$ for $n \leq p$ when $p\geq 5$ and for $n \leq 2$ when $p=3$. His methods are different than ours and Veen does not cover the $p=2$ case.

With this at our disposal, we proceed as in \cite{hesselholt1997polytopesandtruncatedpolynomial}. In \cite{hesselholt1997polytopesandtruncatedpolynomial}, Hesselholt and Madsen compute the algebraic $K$-theory of truncated polynomial algebras by considering them as monoid rings. This provides a decomposition of the cyclic bar construction as a cyclotomic spectrum. Later, these methods are used by Speirs to reproduce these computations using the Nikolaus-Scholze approach to topological cyclic homology \cite{speirs2020truncatedpolynomial}. 

In this work, we generalize this approach to formal DGAs whose homology is a graded monoid ring. In particular, we obtain a decomposition of the cyclic bar construction on $\thh(k)$, which is moreover a specific instance of a functorial decomposition on the $\thh$ of a graded ring spectrum. This
decomposes the Tate and homotopy fixed points spectral sequences, and ultimately, the computation of $\tc(\thh(k),k)$
into simpler computationally accessible pieces. Equipped with both the Dundas-Goodwillie-McCarthy theorem, and Quillen's computation of the algebraic $K$-theory, this results in a complete understanding of the homotopy groups of $\kth(\thh(k))$ for every finite field $k$.

\begin{rem}
It is possible to compute $K_*(\thh(k),k)$ without using Theorem \ref{thm thhfp as an e2 algebra}. Since $\thh(k)$ is equivalent to $k[\Omega S^3]$ as an $E_1$-algebra these $K$-theory groups are given by $K_*(k[\Omega S^3],k)$ and our arguments  provide a computation of these groups. However, with the identification in Theorem \ref{thm e2 dgas with polynomial homology}, the computation of $K_*(\thh(k),k)$ flows in a more direct and natural way. Furthermore, Theorem \ref{thm e2 dgas with polynomial homology} is of interest on its own. For instance, it results in Corollary \ref{cor e1 equivalence of k theories} and  provides the multiplicative structure on the second iterated topological Hochschild homology groups of $k$.
\end{rem}
\subsection{Perspectives}
We now state some applications of these calculations. Note that one has  the following sequence of maps
\[\kth(\kth(\fp)) \to \kth(\tc(\fp)) \to \kth(\thh(\fp)),\]
obtained by applying the $K$-theory functor to the cyclotomic trace and Dennis  maps. Theorem \ref{thm algebraic k theory of thh k} identifies the right hand side. We also identify the middle term by showing that $\tc(\fp)$ as a $\Z_p$-DGA is quasi-isomorphic to the cochain algebra $C^*(S^1, \z_p)$ of the $1$-sphere with coefficients in $\z_p$, see Proposition \ref{prop unique dga with exterior homology}. It follows by the theorem of the heart  \cite[4.8]{antieau2019ktheoreticobstructionstobddtstructures} that these groups are given by $\kth(\z_p)$. Indeed, the sequence above is given by 
\[\kth(\kth(\fp))\to \kth(\z_p) \to \kth(\thh(\fp)).\]
where the second map is induced by a map $H\z_p \to \thh(\fp)$ of $E_\infty$-ring spectra  \cite[\RomanNumeralCaps{4}.4.10]{nikolausscholze2018topologicalcyclic} given by the composite of the connective cover map of $\tc(\fp)$ and the map $\tc(\fp) \to \thh(\fp)$. 
\begin{rem}
It is also important to note that the map $\kth(\Z_p) \to \kth(\thh(\fp))$ as a map of  $E_1$-ring spectra, may or may not be induced by the canonical map
\[H\Z_p\to H\fp \to \thh(\fp).\]
For instance, it follows from  \cite[\RomanNumeralCaps{4}.4.8]{nikolausscholze2018topologicalcyclic} that as a $\T$-equivariant map of $E_\infty$-ring spectra, the map $H\z_p\to \thh(\fp)$ obtained from the connective cover of $\tc(\fp)$ does not factor through $H\fp$. 
\end{rem}

It is interesting to view these maps through the lens of  chromatic redshift. Indeed, the expected behaviour is that the iterated $K$-theory spectrum $\kth(\kth(\fp))$ exhibits $v_2$-periodic phenomena. This is made somewhat more precise in \cite{angelini2018detecting} where it is shown, for $p \geq 5$ is a fixed prime and $q$ a prime power (coprime to $p$) that topologically generates $\Z_p^{\times}$,  that the $\beta$-family in the homotopy groups of spheres is detected by $\kth(\kth(\mathbb{F}_q)_{p})$. This means that under the unit map 
$$
\sph \to \kth(\kth(\mathbb{F}_q)_{p}),
$$
there are classes $\beta_k \in \pi_* \sph$ which map to nonzero elements of $\pi_* \kth (\kth (\mathbb{F}_q)_p)$. This family of elements in the stable homotopy groups of spheres are significant in that they arise from $v_2$ self maps, which is a chromatic height $2$ phenomenon.

In addition, $\thh(\fp)$,  being an $H \Z$-algebra is $K(1)$-acyclic, hence by Mitchell's theorem its $K$-theory is $K(2)$-acyclic. Thus, the map 
$$
\kth(\kth(\fp)) \to \kth(\thh(\fp)) 
$$
annihilates these elements in homotopy. Moreover, note that by recent work in \cite{land2020vanishing}, $K(1)$-localized $\kth$-theory is truncating on $H \Z$-algebras. In particular
$$
L_{K(1)} \kth(\thh(\fp)) \simeq L_{K(1)} \kth(\fp)
$$
so one can say that $K(1)$-localization is insensitive to the difference between the $K$-theory of $\fp$ and of $\thh(\fp)$. In other words, there is an equivalence 
$$
L_{K(1)}\kth(\thh(\fp)); \fp) \simeq L_{K(1)}\kth(\tc(\fp)); \fp)  \simeq 0
$$
in the $K(1)$-local category of spectra. 

Another point of view on Theorem \ref{thm algebraic k theory of thh k} is through what is sometimes called the fundamental theorem of algebraic $K$-theory. This theorem states that for a Noetherian regular ring $R$, the $K$-theory groups of $R$ and $R[x]$ agree \cite[\RomanNumeralCaps{5}.3.3,\RomanNumeralCaps{5}.6.2]{weibel2013k}. Recall that $\thh(k)$ is the formal DGA with homology $k[x_2]$. Therefore, Theorem \ref{thm algebraic k theory of thh k} shows that the fundamental theorem of algebraic $K$-theory does not generalize to formal DGAs. 

\subsection{Future directions}
A potentially fruitful next step would be to generalize these results to the setting where the perfect field $k$ is replaced by a perfectoid ring $R$. Starting off with the fact that 
$$
\pi_* \thh(R) \cong R[v]
$$
for such rings, one would expect the calculations involving the Tate and homotopy fixed point spectral sequences to proceed in a similar fashion.    Alternatively, one way to bypass this issue is to use an alternate argument to compute the Tate and homotopy fixed points, using explicit models for the summands that appear in the decomposition for $\thh(R[ \Omega S^3])$, as they appear in \cite[Section 3.3]{speirs2019ktheorycoordinate} and \cite{riggenbach2020algebraic}, and are in fact exploited in the perfectoid setting in the latter paper. 

Our methods, together with the constructions in \cite{bayindir2019extension}, provide a general framework to study the algebraic $K$-theory of  formal DGAs whose homology is a connected graded monoid. For such DGAs, this results in a decomposition of the cyclic bar construction. We plan to use these ideas to generalize the  calculation of the $K$-theory of truncated polynomial algebras, due to Hesselholt and Madsen, to formal DGAs with truncated polynomial algebra homology. 

Another potential application of the methods used in this paper lies in the realm of $A$-theory calculations. Using the succinct description of $\tc(\sph[\Omega S^3])$ from \cite{bokstedt1993cyclotomictraceandalgebraicktheoryofspaces} (see also \cite{nikolausscholze2018topologicalcyclic}) as a homotopy Cartesian square with terms $\thh(\sph[\Omega S^3])$ and $\Sigma \thh(\sph[\Omega S^3])_{h \T}$ together with our decomposition of the relevant cyclic bar construction, one can compute explicitly the integral homology of $A(S^3)= K( \sph[\Omega S^3])$. We expect to return to these ideas and their extensions in future work. 
\\

\textbf{Outline:} Section \ref{sec e2 identification of thhfp} is devoted to the proof of Theorem \ref{thm thhfp as an e2 algebra} and may be viewed as independent from the rest of this work. Section 3 consists of a brief overview of cyclotomic spectra. Next in section 4, we introduce a graded variant of $\thh$ and use it to obtain functorial decompositions of $\thh(A)$ when $A$ is equipped with a grading. In section 5 we study in greater depth the cyclotomic structure of the spectrum $\thh(\fp[\Omega S^3])$. The pieces are all tied together in section 6 where we compute  $\tp(\thh(\fp)) , \tc^-(\thh(\fp))$ and finally  $\tc(\thh(\fp))$. Last but not least, we include a short argument computing $\kth(\tc(\fp))$ in Section \ref{sec algebraic k theory of tc fp}. 

\textbf{Conventions:} We use $\infty$-categorical language in some of our constructions. However we resort to the  occasional point-set level argument so we have added a short appendix with relevant rectification results.  

\textbf{Acknowledgments}
This work owes an obvious debt to the work of Speirs in \cite{speirs2019ktheorycoordinate, speirs2020truncatedpolynomial}. Furthermore, we are grateful to Martin Speirs for answering our questions on his work. The first author acknowledges support from the project ANR-16-CE40-0003 ChroK. The second author was supported by grant NEDAG ERC-2016-ADG-741501 during the writing of this work. 

\section{A new description of THH$(\fp)$ as an $E_2$-algebra }\label{sec e2 identification of thhfp}
 In this section, we prove Theorem \ref{thm thhfp as an e2 algebra}. In other words, we show that there is an equivalence of $E_2$ $Hk$-algebras
 \[\thh(k) \simeq Hk \wdg \sph[\Omega S^3]\]
 whenever $k$ is a perfect field of positive characteristic. Note that $S^3$ is a group object in spaces and therefore is an $E_1$-algebra. This makes $\Omega S^3$ an $E_2$-algebra in spaces. Together with the fact that the infinite suspension functor is a symmetric monoidal functor, this provides the $E_2$-algebra structure on the right hand side. 
 
 We prove the following stronger statement. 
 \begin{thm} \label{thm e2 dgas with polynomial homology}
 There is a unique connective $E_2$ $Hk$-algebra with homotopy ring $k[x_m]$ for every field $k$ where $\lv x_m\rv = m$ is even.
 \end{thm}
 
To prove this result, we use the theory of $k$-invariants for Postnikov extensions of $E_n$-algebras developed by Basterra and Mandell \cite{basterra2013BP}. We start with the following lemma.
 
 \begin{lem}\label{lem unique exterior dga}
 There is a unique $E_n$ $Hk$-algebra with homotopy ring $\Lambda[x_m]$ (exterior algebra over $k$ with a single generator) for every even $m>0$ and for every $n\geq 0$ including $n=\infty$.
 \end{lem}
 
 \begin{proof}
 For $n=0$, the result follows by the fact that $k$-chain complexes are determined by their homology. For $n>0$, we show that there is a unique Postnikov extension 
 \[X \to Hk\]
where $X$ is an $E_n$ $Hk$-algebra with homotopy ring  $\Lambda[x_m]$.
The $k$-invariants of such extensions are classified by the Andr{\'e}-Quillen Cohomology groups 
\[H^{m+1}_{\mathcal{C}_n}(Hk,Hk)\] 
of $E_n$ $Hk$-algebras \cite[4.2]{basterra2013BP}. Since the input in this cohomology theory is actually our base ring, these groups are trivial. Indeed, to calculate these cohomology groups, one starts with the fiber of the augmentation map of $Hk$. This is the identity map of $Hk$ and therefore this fiber is trivial. It follows by Definition 2.6 of \cite{basterra2013BP} that these cohomology groups are also trivial. 
 \end{proof}

Let $X$ denote the $E_\infty$ $Hk$-algebra corresponding to the formal $E_\infty$ $k$-DGA with homology $k[x_m]$ and let $Y$ be an $E_2$ $Hk$-algebra with homotopy ring $k[x_m]$. We do induction on the Postnikov towers of $X$ and $Y$ in $E_2$ $Hk$-algebras to show that $X$ and $Y$ are weakly equivalent as $E_2$ $Hk$-algebras. Let $X[l]$ and $Y[l]$ denote the degree $l$ Postnikov sections of $X$ and $Y$ respectively. It is sufficient to show that there is a map of Postnikov towers that is a weak equivalence for each $l$. Since all the non-trivial homotopy groups are in degrees that are multiples of $m$, we only need consider the Postnikov sections at multiples of $m$. Note that each Postnikov section of $X$ could also be calculated in $E_\infty$ DGAs and therefore each $X[l]$ also comes from a formal $E_\infty$ $k$-DGA. By Lemma \ref{lem unique exterior dga} above, there is a weak equivalence of $E_2$ $Hk$-algebras, 
\[X[m] \we Y[m].\]
Since $X[0]= Y[0] = Hk$, we can set the degree zero map in the Postnikov tower of $Y$ to be 
\[Y[m] \we X[m] \to Hk.\]
This shows that we have a map of Postnikov section up to degree $m$. 

Now we proceed inductively, assume that we have compatible weak equivalences 
\[Y[l] \we X[l]\]
for every $l \leq m(e-1)$ where $2 \leq e$. It is sufficient to show that there is a weak equivalence $Y[me] \we X[me]$ such that the following diagram commutes. 
\begin{equation} \label{diag map of pstnkv towers}
 \begin{tikzcd}
 Y[me] \ar[r,"\simeq"] \ar[d]& X[me] \ar[d]\\
 Y[m(e-1)]\ar[r,"\simeq"]& X[m(e-1)] 
 \end{tikzcd}
 \end{equation}
 The obstructions to existence of such a lift lie in the Andr{\'e}--Quillen cohomology group
 \[H^{me+1}_{\mathcal{C}_2}(Y[me],Hk).\]
 Indeed, this obstruction class is given by  the pull back of the $k$-invariant of the Postnikov extension 
 \[X[me] \to X[m(e-1)]\]
 through the composition of the left vertical arrow and the lower horizontal arrow, see \cite[4.3]{basterra2013BP}. In order to deduce that this obstruction class is trivial, it is sufficient to show that the map 
 \begin{equation}\label{eq map of aq cohomology groups}
     H^{me+1}_{\mathcal{C}_2}(Y[m(e-1)],Hk) \to H^{me+1}_{\mathcal{C}_2}(Y[me],Hk)
 \end{equation}
is trivial. Note we have 
\[H^{me+1}_{\mathcal{C}_2}(Y[m(e-1)],Hk) = k\]
due to Lemma \ref{lem aq cohomology group calculation}. The $k$-invariant of the Postnikov extension 
\[Y[me] \to Y[m(e-1)]\]
also lies in this group and this $k$-invariant is non-trivial; otherwise, the homotopy ring of $Y[me]$ would be a square-zero extension \cite[4.2]{basterra2013BP}. Let $x\in k$ denote this $k$-invariant. 

There is a pullback square 
\begin{equation*} 
 \begin{tikzcd}
 Y[me] \ar[r] \ar[d]& Y[m(e-1)] \ar[d,"i"]\\
 Y[m(e-1)]\ar[r,"x"]& Y[m(e-1)] \vee \Sigma^{me+1}Hk.
 \end{tikzcd}
 \end{equation*}
where $i$ denotes the trivial derivation. This shows that the left vertical map in both diagrams above pulls back $x$ to the trivial element. In other words, $x$ is mapped to $0$ by the map given in \eqref{eq map of aq cohomology groups}. Since this is a map of $k$-modules and $x \neq 0$ in the field $k$, we deduce that the map in \eqref{eq map of aq cohomology groups} is trivial. Hence, the obstruction to the lift shown in Diagram \eqref{diag map of pstnkv towers} is trivial.

This shows that there is a map $Y[me] \to X[me]$ that makes  Diagram \eqref{diag map of pstnkv towers} commute up to homotopy. By inspection on homotopy groups, it is clear that this map is a weak equivalence. 

In order to obtain a weak equivalence $Y[me] \we X[me]$ that makes Diagram \eqref{diag map of pstnkv towers} commute strictly, we could start with Postnikov towers where the objects are cofibrant and the section maps are fibrations, see the proof of Theorem 4.2 in \cite{basterra2013BP}. In this situation, one argues as in the proof of Proposition A.2 in \cite{bayindir2018topeqeinfty} to show that $Y[me]$ and $X[me]$ are weakly equivalent in the category of $E_2$ $Hk$-algebras over $X[m(e-1)]$ and obtain the desired weak equivalence by noting that $Y[me]$ is cofibrant and $X[me]$ is fibrant in this category. 

What is left to prove is the following lemma.

 \begin{lem}\label{lem aq cohomology group calculation}
 Let $Z$ be the $E_2$ $Hk$-algebra corresponding to the formal $E_\infty$ $k$-DGA with homology $k[x_m]/(x_m^{e})$. In this situation, we have 
\[ H^{me+1}_{\mathcal{C}_2}(Z,Hk) \cong k\]
 \end{lem}
 \begin{proof}
 Using the universal coefficient spectral sequence of \cite[3.1]{basterra2013BP}, one sees that this group is the $k$-dual of the Andre Quillen homology group  $H_{me+1}^{\mathcal{C}_n}(Z,Hk)$. Therefore, it is sufficient to show that 
 \[H_{me+1}^{\mathcal{C}_2}(Z,Hk)=k.\]
 By Theorem 3.2 of \cite{basterra2013BP}, there is spectral sequence calculating 
  \[B^1_{*} = k \oplus H^{\mathcal{C}_1}_{*-1}(Z,k).\]
 whose second page is given by 
 \[E^2_{*,*} = \hh^{k}_{*,*}(k[x_m]/(x^{e}),k).\]
 The identification of the second page follows from the description of the bar construction given in \cite[Section 6]{basterra2011thhofEn}.
 We have
\[E^2_{*,*} = \Lambda(\sigma x_m) \otimes \Gamma(\varphi^e x_m)\]
where $\text{deg}(\sigma x_m)= (1,m)$, $\text{deg}(\varphi^e x_m) = (2,em)$  and $\Gamma(\varphi^e x_m)$ denotes the divided power algebra over $k$ on a single generator, see \cite[5.9]{bayindir2019dgaswithpolynomial}. The differentials in this spectral sequence are trivial for degree reasons. We obtain an isomorphism of $k$-modules 
\[B^1_* = \Lambda(y) \otimes \Gamma(z)\]
where $\lv y \rv = m+1$, $\lv z \rv = me+2$.

Note that $Z$ is an $E_\infty$ $Hk$-algebra by hypothesis because it corresponds to a commutative $k$-DGA. Therefore,  $B^1_*$ is a graded commutative ring \cite[Section 3]{basterra2013BP}. From this, we obtain that $y^2=0$.  

There is another spectral sequence with 
\[E^2_{*,*} = \textup{Tor}^{B^1_*}_{*,*}(k,k)\]
converging to 
 \[B^2_{*} = k \oplus H^{\mathcal{C}_2}_{*-2}(Z,k),\]
 see \cite[3.3]{basterra2013BP}.  We only care about this page up to total degree $me+3$. Therefore we can safely assume a ring isomorphism 
 \[B^1_* \cong \Lambda(y) \otimes k[z].\]
  We obtain that
 \[E^2_{*,*} =  \textup{Tor}^{\Lambda(y)}_{*,*}(k,k) \otimes \textup{Tor}^{k[z]}_{*,*}(k,k) \cong \Gamma(\sigma y) \otimes \Lambda(\sigma z).\]
 where $\text{deg}(\sigma y)= (1,m+1)$ and $\text{deg}(\sigma z)= (1,me+2)$. Note that all divided  powers of $\sigma y$ are in even total degrees and therefore they do not contribute to total degree $me+3$. This is the degree we care about and the only contribution to this degree comes from $\sigma z$. Furthermore, $\sigma z$ survives to the $E^\infty$ page because of degree reasons. Therefore we have either 
 \[H^{me+1}_{\mathcal{C}_2}(Z,Hk)= k \textup{\ or \ } H^{me+1}_{\mathcal{C}_2}(Z,Hk)= 0.\]
 If the second option is true, then we would conclude that $Z$ has a unique Postnikov extension with homotopy groups isomorphic to $k[x_m]/(x^{e+1})$ as a $k$-module. This follows by Theorem 4.2 of \cite{basterra2013BP}. We know that there are at least two distinct Postnikov extensions of $Z$ of this type. One is the square-zero extension which  leads to the square-zero extension ring
 \[k[x_m]/(x_m^{e})\oplus \Sigma^{me}k\]
 in homotopy.
 The other Postnikov extension of $Z$ is the $E_2$ $Hk$-algebra corresponding to the formal $E_\infty$  $k$-DGA with homology $k[x_m]/(x^{e+1})$. Note that this Postnikov extension has a homotopy ring that is not isomorphic to the homotopy ring of the square-zero extension. This shows that the second option above is not true. 
 \end{proof}

\section{Trace methods}

We give in this section a quick recap on the  Nikolaus-Scholze approach to cyclotomic spectra. 

\subsection{Cyclotomic spectra and topological cyclic homology}

As discussed earlier, our $K$-theory calculations reduce to topological cyclic homology calculations as a consequence of the results of Dundas, Goodwillie and McCarthy. 

For our  calculations of topological cyclic homology, we use the recent Nikolaus-Scholze definition of topological cyclic homology and cyclotomic spectra. Note that for a spectrum $X$ with a $C_p$-action, the Tate construction $X^{tC_p}$ denotes the cofiber of the norm map $X_{hC_p} \to X^{hC_p}$, see \cite[\RomanNumeralCaps{1}.1]{nikolausscholze2018topologicalcyclic}.

\begin{defn}\cite[1.3]{nikolausscholze2018topologicalcyclic}
A cyclotomic spectrum is a spectrum $X$ with a $\T$-action and $\T$-equivariant maps $\varphi \co X \to X^{tC_p}$ for every prime $p$ where the $\T$-action on the right hand side is given by the residual $\T/C_p\cong \T$-action. These maps are called the Frobenius maps of $X$. 
\end{defn}

This definition agrees with the older B\"okstedt--Hsiang--Madsen definition of cyclotomic spectra for cyclotomic spectra that are bounded below in homotopy \cite[1.4]{nikolausscholze2018topologicalcyclic}. The bounded below
hypothesis is satisfied by all the spectra that we consider in this work.

The main examples of cyclotomic spectra are the  topological Hochschild homology of  ring spectra. Indeed,  $\thh(-)$ is a symmetric monoidal functor \begin{equation} \label{eq thh is symmetric monoidal}
    \textup{Alg}_{E_1}(\textup{Sp}) \to \textup{Cyc Sp,}
\end{equation}
see \cite[Section \RomanNumeralCaps{4}.2]{nikolausscholze2018topologicalcyclic}.

With this structure, the topological cyclic homology of a connective ring spectrum $A$ is defined via the following fiber sequence.

\[\tc(A) \to \thh(A)^{h\T} \xrightarrow{\prod_{p \in \mathbb{P}}(\varphi^{h\T}-can)} \prod_{p \in \mathbb{P}} (\thh(A)^{tC_p})^{h\T}\]
Here, $can$ is the canonical map obtained by the identification of the middle term as $(\thh(A)^{hC_p})^{h\T/C_p}$.

Note that if $p$ is invertible on $A$, then $\thh(A)^{tC_p}$= 0. This simplifies the definition of cyclotomic spectra as well as the definition of topological cyclic homology in various situations. For instance, only one of the factors of the product above is possibly nontrivial when $A$ is an $H\fp$-module.

\section{$\thh$ of graded ring spectra}\label{gradedthh}
The results of this section help explain and place in a general context one of the key ingredients to our computation of $K(\thh(\fp))$, namely that of a weight decomposition of $\thh(\thh(\fp))$ into homologically more tractable $\T$-equivariant spectra, thus facilitating the trace methods used here and in \cite{hesselholt1997polytopesandtruncatedpolynomial, speirs2019ktheorycoordinate,speirs2020truncatedpolynomial} to compute $K$-theory in various settings.

In particular, we introduce the $\infty$-category of $\Z$-graded spectra and describe variant of $\thh$ for $E_1$-algebras in this category. We then use it to show that when $A$ is a ring spectrum admitting a grading, then $\thh(A)$ inherits a canonical grading.  
We would like to point out that more general results along the lines of those obtained in this section also appear in \cite[Appendix A]{antieau2020beilinson}.
\begin{defn}
Let $\mathcal{C}$ be an $\infty$-category. The category of graded objects of $\mathcal{C}$ is
$$
\mathcal{C}^{\Z}:= \on{Fun}(\Z, \mathcal{C}) 
$$
where $\Z$ is the integers regarded as a discrete $\infty$-groupoid. 
\end{defn}

\begin{defn}
Let $\C^{\otimes}$ be a symmetric monoidal $\infty$-category. Then $\C^\Z$ inherits 
a symmetric monoidal structure, the so-called Day convolution product from that of $\C^{\otimes}$ and from the abelian group structure on $\Z$. Concretely, if $X,Y \in \C^\Z$, then $X \otimes Y$ is given by
$$
(X \otimes Y)(n) \simeq  \oplus_{i+j=n} X(i) \otimes Y(j) 
$$
(see \cite[Section 2.3]{lurie2015rotation} for a more comprehensive account.)
\end{defn}

In order to prove the main proposition of this section, we briefly introduce the monoidal envelope construction appearing in \cite{nikolausscholze2018topologicalcyclic} (see also \cite[2.2.4.1]{lurie2016higher})

\begin{defn}
Let $\mathcal{C}^\otimes$ be a symmetric monoidal $\infty$-category. The monoidal envelope  of $\mathcal{C}^\otimes$ is the fiber product
$$
\C^{\otimes}_{\on{act}}:= \mathcal{C}^\otimes \times_{\on{Fin}_*} \on{Fin}
$$ 
where $\on{Fin} \subset \on{Fin}_*$ is the subcategory of the category of based finite sets consisting of the active morphisms. In this setting, a morphism $f: [n] \to [m]$ of based finite sets is active if $f^{-1}(*)= *$.
\end{defn}
Note that the fiber of $\C^{\otimes}_{\on{act}}$ over a finite set $I$ is just $\C^I$, so that objects may be thought of as lists of objects of $\C$. We remark that this comes equipped with a canonical symmetric monoidal structure 
$$
\oplus: \C^{\otimes}_{\on{act}} \otimes \C^{\otimes}_{\on{act}} \to \C^{\otimes}_{\on{act}}  \, \, \, \, \, \,  ((X_i)_{i \in I}, (X_j)_{j \in J})\mapsto (X_k)_{k \in I \sqcup J},
$$
together with a natural symmetric monoidal functor 
$$
\iota_{\C}: \C^{\otimes} \to \C^{\otimes}_{\on{act}}.
$$
By \cite[Remark 3.4]{nikolausscholze2018topologicalcyclic}, this admits a left adjoint
$$
\otimes: \C^{\otimes}_{\on{act}} \to \C^\otimes
$$

\begin{prop} \label{splittingofthh}
The functor  
$$
\thh: \on{Alg}_{\on{Sp}} \to \on{Sp}^{B \T}
$$
may be promoted to a functor 
$$
\widetilde{\thh}: \on{Alg}_{\on{Sp}^\Z}  \to \on {Sp}^{\Z \times B \T},
$$
where $\on{Sp}^{\Z \times B\T}$ denotes the $\infty$-category of graded spectra with a $\T$-action. 
Moreover, there exists a natural equivalence
$$
\colim_{\Z }\widetilde{\thh(-)} \simeq \thh(-)
$$
\end{prop}

\begin{proof}
We briefly recall the construction of $\thh$ as it appears in \cite{nikolausscholze2018topologicalcyclic}. Given an associative algebra $A$, the spectrum $\thh(A)$ is defined as the realization of the following cyclic object 
$$
\Lambda^{op} \to \assact \xrightarrow{A^{\otimes}} \spact \xrightarrow{\otimes} \on{Sp}
$$
The first functor is defined in \cite[Proposition B.1]{nikolausscholze2018topologicalcyclic}. Now let $(A(n))_{n \in \Z}$ be a graded $E_1$-algebra (an object in $\on{Alg}_{\on{Sp}^\Z}$) such that 
$\colim_\Z A(n) \simeq \bigvee_{n \in \Z} A(n) \simeq A$. We define $\widetilde{\thh(A)}$ to be the realization of the cyclic object in $\on{Sp}^{\Z}$ classified by the following diagram
$$
\Lambda^{op} \to \assact \xrightarrow{(A(n))_{n \in \Z}^\otimes} \spactprod \to \on{Sp}^\Z.
$$
We now show that the following diagram is commutative. 
\begin{equation} 
 \begin{tikzcd}
\Lambda^{op} \ar[r,""] \ar[d, "=" ]& 
\assact \ar[d, "="] \ar[r,""]& \spactprod  \ar[d, "({\colim_\Z})^\otimes_{\on{act}}"] \ar[r, ""]& \on{Sp}^\Z \ar[d, "\colim_\Z"]&
\\
 \Lambda^{op} \ar[r,""]&  \assact \ar[r,""]& \spact \ar[r,""]& \on{Sp}.
 \end{tikzcd}
 \end{equation}
Keeping in mind the fact that 
$$
\colim_\Z: \on{Sp}^\Z \to \on{Sp}
$$
is symmetric monoidal, the commutativity of the middle square follows by applying the pullback along the map $\on{Fin} \to \on{Fin}_*$
to the commutative square of $\infty$-operads  
$$
   \begin{tikzcd}
     \on{Ass} \ar[r,"A(n)"] \ar[d, "=" ]& \on{Sp}^{\Z} \ar[d, "\colim_\Z"]\\
    \on{Ass} \ar[r,"A"] & \on{Sp}^\otimes,
    \end{tikzcd}
$$
classifying the compatible associative algebra structure on $(A(n))_{n \in \Z}$ and $A$; the commutativity of the right square follows from Lemma   \ref{commutativityofactshit}
below. 
Hence, we have shown that one can apply the ``cyclic bar construction" to a graded monoid in spectra to obtain a cyclic object in graded spectra. We remark that $\on{Sp}^{\Z}$ is a product in $\widehat{Cat_\infty}$\footnote{We use the notation $\widehat{Cat_\infty}$ to denote the $\infty$-category of (not necessarily small) $\infty$-categories, cf. \cite[Chapter 3]{lurie2009higher}} (in fact, for each $n \in \Z$, one has a retract $\on{Sp}^{\Z} \to \on{Sp}$ see e.g. \cite[Proof of Proposition 4.2]{moulinos2019geometry}); hence a cyclic object  in $\on{Sp}^\Z$ is equivalent to a cyclic object in $\on{Sp}$ for every integer $n$. 
\end{proof}

We have used the following lemma in the above proof:

\begin{lem}\label{commutativityofactshit}
Suppose $F: \C^\otimes \to \mathcal{D}^\otimes$ is a symmetric monoidal functor. Then the following diagram commutes 
\begin{equation*}
    \begin{tikzcd}
     \C^{\otimes}_{\on{act}} \ar[r,"F_{\on{act}}"] \ar[d, "\otimes" ]& \mathcal{D}^{\otimes}_{\on{act}} \ar[d, "\otimes"]\\
    \C^{\otimes} \ar[r,"F"] & \mathcal{D}^\otimes
    \end{tikzcd}
\end{equation*}
\end{lem}

\begin{proof}
As described in \cite[Proposition III.3.2]{nikolausscholze2018topologicalcyclic}, there is an equivalence 
$$
\on{Fun}_{lax}(\C, \mathcal{D}) \simeq \on{Fun}^{\otimes}( \C^{\otimes}_{\on{act}}, \mathcal{D}) 
$$
where the left hand side denotes the $\infty$-category of operad maps and the equivalence is induced by restriction along the canonical lax monoidal functor $\iota_C: \C^\otimes \to \C^{\otimes}_{\on{act}}$. Moreover, the symmetric monoidal left adjoint $\otimes: \C^{\otimes}_{\on{act}} \to \C$ can be characterized as the image of the identity functor $id \in \on{Fun}_{lax}(\C, \mathcal{C})$ under this equivalence (\cite[Remark III.3.4]{nikolausscholze2018topologicalcyclic}) From this we deduce that 
\begin{equation}\label{compatibilitything}
    \otimes \circ \iota \simeq id
\end{equation}
for any symmetric monoidal category $\C$. 
Hence, the diagram in the statement commutes if and only if it commutes upon precomposition by $\iota_\C: \C^{\otimes} \to \C^{\otimes}_{\on{act}}$:
$$
\otimes_{\mathcal{D}} \circ F_{\on{act}} \circ \iota_C  \simeq   F \circ \otimes_C \circ \iota_C. 
$$
This is true since $\otimes_\C \circ \iota_\C \simeq id_\C, \otimes_\mathcal{D} \circ \iota_\mathcal{D} \simeq id_{\mathcal{D}}$ (by \ref{compatibilitything}), together with the fact that $ F_{\on{act}} \circ \iota_C  \simeq \iota_D \circ F$ (by naturality of $\iota_C$).
\end{proof}

As a corollary, we obtain, a canonical splitting on the topological Hochschild homology of a graded algebra $A$.

\begin{cor}\label{cor grading in thh from a graded ring spectrum}
Let $A$ be a graded $E_1$-algebra, so that there exists an algebra
$(A(n))_{n\in \Z}$ in $\on{Sp}^{\Z}$ with $\colim_\Z (A(n)) \simeq  A$. Then $\thh(A)$ admits a canonical $\T$-equivariant decomposition 
$$
\thh(A) \simeq \bigvee_{m \in \Z} B(m)
$$
\end{cor}

\begin{proof}
This follows from Proposition  \ref{splittingofthh} since the cyclic object realizing to $\thh(A)$ splits as a cyclic object in spectra for each $n \in \Z$, together with the fact (eg. \cite[Proposition B.5]{nikolausscholze2018topologicalcyclic}) that upon applying realization, this results in splitting of objects in $\on{Sp}^{B \T}$.  
\end{proof}

From now on, we will let $B(-)$ denote the components of $\thh(A)$ guaranteed by Corollary \ref{cor grading in thh from a graded ring spectrum} for a given graded $E_1$-algebra $A$. The decomposition is compatible with the Frobenius maps 
$$
\varphi: \thh(A) \to \thh(A)^{t C_p}
$$
in the following manner:

\begin{prop}\label{lem frobenius also split}
Let $A$ be a graded $E_1$-algebra. For every prime p, the $p$-typical Frobenius map on $\thh(A)$ restricts to maps \[\varphi_m \co B(m) \to B(pm)^{tCp}\]
on the summands.
\end{prop} 

\begin{proof}
We recall from \cite{nikolausscholze2018topologicalcyclic} the construction of the maps $\thh(A) \to \thh(A)^{t C_p}$. The key ingredient is a natural transformation{\color{red}}, from the functor 
$$
I: N(\on{Free}_{Cp}) \times_{N(\on{Fin})} \on{Sp}_{\on{act}}^{\otimes} \to \on{Sp}
$$
given heuristically by 
$$
(S, (X)_{\overline{s} \in S/C_p}) \mapsto \bigotimes_{\overline{s} \in S/C_p}X_{\overline{s}}
$$
to the functor 
$$
\tilde{T}_p: N(\on{Free}_{Cp}) \times_{N(\on{Fin})} \on{Sp}_{\on{act}}^{\otimes} \to \on{Sp}
$$
given by 
$$
(S, (X)_{\overline{s} \in S/C_p})  \mapsto (\bigotimes_{s \in S}X_{s})^{t C_p}. 
$$
This makes precise the following natural transformation of cyclic objects:

\begin{equation} 
 \begin{tikzcd}
... \ar[r, shift left =2] \ar[r,shift right ] \ar[r, shift left ] \ar[r, ""] & 
A^{\wedge 3} \ar[d, "\Delta_p"]  \ar[r,shift right ] \ar[r, shift left ] \ar[r, ""] & A^{\wedge 2}  \ar[d, "\Delta_p"] \ar[r, shift right] \ar[r, ""]& A \ar[d, "\Delta_p"]&
\\
... \ar[r, shift left =2]  \ar[r, shift left ] \ar[r, ""]  \ar[r,shift right] \ar[r, ""]&  (A^{\wedge 3p})^{t C_p} \ar[r, shift left ] \ar[r, shift right] \ar[r,""]& (A^{\wedge 2p})^{t C_p}  \ar[r, shift right] \ar[r,""]& (A^{\wedge p})^{t C_p}
 \end{tikzcd}
 \end{equation}
 
The existence of this natural transformation follows from \cite[Lemma III.3.7]{nikolausscholze2018topologicalcyclic}.
Note that in each cyclic degree, the map is precisely the Tate diagonal. By naturality of the Tate diagonal map the following diagram commutes

\begin{equation} 
 \begin{tikzcd}
 B_n(m) \ar[r,""] \ar[d, "\Delta" ]& \bigvee_m B_n(m) \simeq A^{\wedge n} \ar[d, "\Delta"]\\
 (B_n(m)^{\otimes p})^{t C_p} \ar[r,""]&  ((\bigvee_m   B_n(m))^{\otimes p})^{tC_p} \simeq (\bigvee_m B_{np}(m))^{t C_p} \simeq (A^{\wedge np})^{t C_p}
 \end{tikzcd}
 \end{equation}
where $B_n(m)$ denotes the $n$th simplicial level of the cyclic object $B_{\bullet}(m)$. The top horizontal row represents the inclusion, in simplicial degree $n$ of the summand 
$B(m) \to \thh(A)$.
In particular, the right hand map is the natural transformation $I \to \tilde{T}_p$ aluded to above in simplicial degree $n-1$.  Finally, the  bottom horizontal map is the Tate fix points functor $(-)^{tC_p}$ applied to 
the inclusion map of the summand 
$$
B_n(m)^{\wedge p } \to B_{np}(mp)= \bigvee_{i_1 +... +i_{p}= mp}B_n(i_1)\otimes...\otimes B_n(i_p),
$$
where $i_1=...=i_p=m$. Here, $B_{np}(mp)$ may be understood as the $np$-th level of the cyclic object which upon taking realization recovers $B(pm)$. 
We remark that $B_n(m)^{\wedge p}$ has to map into the summand $B_{np}(pm)$. Indeed, as the multiplicative structure on $A$ is induced from that of an algebra in graded spectra (see section \ref{gradedthh}), the $p$-fold multiplication of summands in weight $m$ of  $A^{\wedge n}$ sends these to summands in weight $pn$ in $A^{\wedge pn}$. Hence we obtain, for each $m$, a morphism of cyclic objects 
$$
\varphi: B_\bullet(m) \to B_{\bullet}(pm)^{t C_p}
$$
Taking into account the compatibility map between the realization of the simplicial diagram of Tate fix points to the Tate fix points of the realization
$$
|B_{\bullet}(pm)^{tC_p}| \to B(pm)^{tC_p}, 
$$
we obtain the desired description of the restriction of the Frobenius \[
\varphi_m: B(m) \to B(pm)^{t C_p}.
\]
\end{proof}

\subsection{Compatibility with the pointed cyclic bar construction}
We also consider monoid objects in the $\infty$-category of based spaces $\mathcal{S}_*$. To these, one may apply a pointed version of the cyclic bar construction, as it appears in \cite[Section 3.1]{speirs2019ktheorycoordinate}. In the situations that arise in this paper, these pointed monoids will have an additional structure, that of a wedge decomposition 
$$
A \simeq \bigvee A(n)
$$
giving them the structure of a graded object in pointed spaces. Applying the cyclic bar construction to $A$ results in a graded cyclic object in pointed spaces. Let $b_\bullet$ denote the cyclic bar construction of $A$ and assume that $b_\bullet$ is proper. In this situation, the ``weight" $m$ part of $b_\bullet$ at cyclic degree $s$ is given by 
\[b_s(m) = \bigvee_{ n_1 +n_2+ \cdots +n_{s+1} = m} A(n_1) \wdg \cdots \wdg A(n_{s+1}).\]
We have the following compatibility result between the pointed cyclic bar construction of graded monoids in pointed topological spaces and the THH of graded monoids in the $\infty$-category of spectra. 

\begin{prop} \label{prop compatibility between realization in spectra and pointed simplicial sets}
Following the notation above, we have the following equivalence in the $\infty$-category of  $\T$-spectra.
\[\Sigma^\infty \lv b_\bullet(m) \rv \simeq  \widetilde{\thh}(\Sigma^\infty A)(m)  \]
Here, $\Sigma^\infty-$ denotes the infinite suspension functor from pointed spaces to the $\infty$-category of spectra and $\lv - \rv$ denotes the geometric realization functor in pointed topological spaces. 
\end{prop}
\begin{proof}
Let $A$ be an $E_1$-monoid in the $\infty$-category of graded pointed spaces, giving rise via infinite suspension to the graded $E_1$-algebra $\Sigma^\infty A$. Since $\Sigma^\infty(-)$ is compatible with the relevant symmetric monoidal structures and preserves coproducts, one obtains a decomposition
$$
\thh(\Sigma^\infty A) \simeq  \Sigma^\infty( \vee_m |b_\bullet(m)|) \simeq  \vee_m \Sigma^\infty |b_\bullet(m)|\simeq \vee_m |\Sigma^\infty b_\bullet(m)|,
$$
where the first equivalence follows for example from \cite[Theorem 7.1]{hesselholtwittvectors1997k}. 
By inspection, the cyclic objects $\Sigma^\infty b_\bullet(m)$ are precisely the weight $m$ cyclic summand realizing to 
$\widetilde{\thh}(\Sigma^\infty A)(m)$ in the decomposition of corollary \ref{cor grading in thh from a graded ring spectrum} since infinite suspension at the level of cyclic objects preserves internal weight $m$. 
\end{proof}

\section{Weight decomposition on topological Hochschild homology}

 For the rest of this section, let $X$ denote $\thh(k)$ for a perfect field  $k$  of characteristic $p$. We describe in more detail the  weight decomposition on $\thh(X)$ and describe its equivariant properties. This is analogous to the weight splitting used by Hesselholt and Madsen for the  Hochschild homology of quotients of discrete polynomial rings on a single generator \cite{hesselholt1997polytopesandtruncatedpolynomial}. 
 
 \subsection{A $\T$-equivariant splitting of $\thh(\thh(\fp))$}
We start by describing the weight decomposition for $\thh(X)$. Due to Theorem \ref{thm thhfp as an e2 algebra}, there is an equivalence of $E_2$-algebras. 
\[X \simeq Hk \wdg \sph[\Omega S^3].\]
This, together with \eqref{eq thh is symmetric monoidal} shows that we have the following splitting of  $E_1$-algebras in $\T$-spectra.

 \[\thh(X) \simeq \thh(k) \wdg \thh(\sph[\Omega S^3]) \]
 
 We now construct a weight decomposition on $\thh(\sph[\Omega S^3])$. For this, note that there is an equivalence of $E_1$-algebras (see for example \cite{may1978splitting})
  \begin{equation} \label{eq free e1 algebra on 2 sphere}
    \sph[\Omega S^3] \simeq \bigvee_{0 \leq m} \Sigma^{2m} \sph,
\end{equation}
making $\sph[\Omega S^3]$ into a graded spectrum where the degree $m$ part is given by $\Sigma^{2m} \sph$.  Here, $\Sigma^{2m}\sph$ denotes $(\Sigma^{2}\sph)^{\wdg m}$ for $m\geq1$, $\sph$ for $m=0$ and $\Sigma^2\sph$ denotes the double suspension of $\sph$. With these identifications, the product on the right hand side is given by the concatenation of the smash factors.

One may also obtain this graded spectrum from the monoid in graded pointed spaces given by 
\[\bigvee_{0 \leq m} \Sigma^{2m} S\]
via applying the infinite suspension functor. Here, $\Sigma^2S$ denotes the $2$-sphere as a pointed simplicial set and the rest of the cofactors are given by smash powers as before.

By Proposition \ref{splittingofthh}, $\thh(\sph[\Omega S^3])$ is canonically a graded object in $\T$-spectra. We analyze this splitting more closely. Let $B = \thh(\sph[\Omega S^3])$ and 
\[B(m) = \widetilde{\thh}(\sph[\Omega S^3])(m).\]
Recall that the standard complex realizing $B$ is  given by 
\[B_s = \sph[\Omega S^3]^{\wdg{s+1}}\]
at simplicial degree $s$. Therefore, we have 
\[B_s \cong \bigvee_{0 \leq m_1} \bigvee_{0 \leq m_2} \cdots \bigvee_{0 \leq m_{s+1}} \Sigma^{2m_1} \sph \wedge \Sigma^{2m_2} \sph \wdg \cdots \wdg \Sigma^{2m_{s+1}} \sph.\]
With the previous identifications,  $B(m)$ is given by the realization of the subcomplex $B_\bullet(m)$ consisting of the cofactors with   
\[m_1 + m_2 + \cdots +  m_{s+1} = m.\]

Therefore, we obtain the following splitting of $\thh(\sph[\Omega S^3])$ as a spectrum with  $\T$-action.
\[\thh(\sph[\Omega S^3]) \simeq \bigvee_{0\leq m}B(m)\]
 This gives the following splitting of spectra with $\T$-action.
\begin{equation}
 \label{eq thhx splits as a sum}
\thh(X) \simeq \thh(k) \wedge \bigvee_{0\leq m}B(m)
\end{equation}

\begin{rem}
We remark that splittings for $\thh(\sph[\Omega \Sigma U])$ for $U$ a  connected space have been previously obtained. For example, by Cohen in \cite{cohen1987model}, there is a decomposition
$$
\Sigma^{\infty}\mathcal{L} \Sigma U \simeq \bigvee_{n \geq 1} \Sigma^\infty(U^n \wedge_{C_n} \T_+)  
$$
It is not clear to the authors whether or not this splitting is $\T$-equivariant and therefore whether it coincides with the decomposition described above. 
\end{rem}
\subsection{Homology of the summands}
For our calculations, we need to understand the homology of $B$ and its summands. Note that we have 
\begin{equation*}
    \begin{split}
        H\Z \wdg B \simeq& H\Z \wdg \thh(\sph[\Omega S^3]) \\
                  \simeq& \hh(H\z \wdg \sph[\Omega S^3])      
    \end{split}
\end{equation*}

Due to equation \eqref{eq free e1 algebra on 2 sphere}, this is the Hochschild homology of the free DGA with a single generator in degree $2$. The $E^2$ page of the standard spectral sequence calculating these Hochschild homology groups is given by
\[E^2 \cong \tor_{*,*}^{\z[x_2]^e}(\z[x_2],\z[x_2])\]
where $\z[x_2]^e$ is the enveloping algebra $\z[x_2]\otimes \z[x_2]$ and the subscript denotes the internal degree \cite[\RomanNumeralCaps{9}.2.8]{elmendorf2007rings}. The standard argument is to use the automorphism of $\z[x_2]^e$ given by 
\[x_2 \otimes 1 \to x_2\otimes 1 - 1\otimes x_2 \textup{\ and \ } 1\otimes x_2 \to 1\otimes x_2.\]
Precomposing with this automorphism makes the action of the first factor of $\z[x_2]^e$ trivial and the action of the  second factor the canonical non-trivial action on $\z[x_2]$. We obtain 
\begin{equation*}
E^2 \cong \z[x_2] \otimes  \tor_{\z[x_2]}(\Z,\Z)\cong  \z[x_2] \otimes  \Lambda(\sigma x_2) 
\end{equation*}
where $\text{deg}(\sigma x_2) = (1,2)$. The differentials on this spectral sequence are trivial due to degree reasons. We obtain an isomorphism of abelian groups
\[H\z_* B \cong \Z[x_2] \otimes \Lambda(y_3)\]
where the $\lv y_3 \rv = 3$.

\begin{rem}
Alternatively, one may deduce the above by using the well-known identification 
$$
\thh(\sph[\Omega S^3]) \simeq  \Sigma^\infty_+ \mathcal{L} X
$$
togethere with the fact that $S^3$ is an $H$-space (it is in fact a Lie group) which gives a splitting of the canonical loop space fiber sequence 
$$
\Omega S^3 \to \mathcal{L} S^3 \to S^3
$$
Hence, $ \mathcal{L} S^3 \simeq \Omega S^3 \times S^3$ and so, there is an equivalence 
$$
H \Z[\mathcal{L}S^3] \simeq H \Z [\Omega S^3] \otimes H \Z[S^3]
$$
Upon taking homotopy, groups, and using the fact that the integral homology of $\Omega S^3$ is polynomial in degree $2$, one obtains the above isomorphism. 
\end{rem}

By inspection of the simplicial resolution defining this spectral sequence, $H\z_* B(m)$ is given by the contribution that comes from the internal degree $2m$ elements in the spectral sequence above. We obtain the following. 

\begin{lem} \label{lem Z homology of Bm}
The $H\Z$ homology of $B(m)$ is given by $\Z$ concentrated in degrees $2m$ and $2m+1$.
\end{lem}

We now turn our attention to $Hk_*B(m)$ and its multiplicative structure. This will in turn, be used to recover the multiplicative structure on the topological Hochschild homology groups of $\thh(k)$, see \eqref{eq thhthhfp ring structure}.

Before we proceed, we prove the following compatibility statement between $\thh(A)$ and $\hh^{k}(A \wedge Hk)$.  Surely this is well known, but we were unable to find a reference.  

\begin{prop}\label{Enalgebras}
Let $A$ be an $E_{n}$-ring spectrum. There is a natural equivalence of $E_{n-1}$-algebras: 
$$
\thh(A) \wedge Hk \simeq \hh^{k}(A \wedge Hk)
$$
in $\Mod_{Hk}$. 
\end{prop}
\noindent 
We shall need the following lemma:
\begin{lem}
Let $F$ be a symmetric monoidal left adjoint functor $ \on{Sp} \to \Mod_{Hk}$. There is an induced natural equivalence of symmetric monoidal functors 
$$
\colim_{\Delta^{op}} \circ F \to F \circ \colim_{\Delta^{op}}: \on{Sp}^{\Delta^{op}}  \to \Mod_{Hk}
$$
\end{lem}

\begin{proof}
The map realizing this equivalence is the standard natural transformation
\begin{equation} \label{colimitexchangen}
   \colim_{\Delta^{op}} F\circ (X_\bullet) \to F \circ \colim_{\Delta^{op}}(X_\bullet), 
\end{equation}
of functors $\on{Sp}^{\Delta^{op}} \to \on{Sp}$
induced by the universal property of the colimit. Since $F$ commutes with colimits, this is an equivalence.  We now show that this is a multiplicative natural transformation. To do this we interpret it as a  map in the $\infty$-category of functors $\on{Fun}(\on{Sp}^{\Delta^{op}}, \on{Sp})$. Recall that the category of lax symmetric monoidal functors is identified with the category
$$
\on{CAlg}(\on{Fun}(\on{Sp}^{\Delta^{op}}, \Mod_{Hk}))
$$
of commutative algebra objects with respect to the Day convolution monoidal structure on $\on{Fun}(\on{Sp}^{\Delta^{op}}, \Mod_{Hk})$. Hence the problem reduces to displaying (\ref{colimitexchangen}) as a map of commutative algebra objects in this category. For $X_\bullet \in \on{Sp}^{\Delta^{op}}$, this is a colimit of maps
$$
F(ev_{n}(X_\bullet)) \to F(\colim_{\Delta^{op}}X_\bullet),
$$
each of which is a  map of commutative algebra objects in $\on{Fun}(\on{Sp}^{\Delta^{op}}, \Mod_k)$.
Since sifted colimits of (commutative) algebra objects are detected in the underlying $\infty$-category (see \cite[Corollary 3.2.3.2]{lurie2016higher}) and since $F$ is itself symmetric monoidal, we may identify (\ref{colimitexchangen}) as the map induced by the universal property of the colimit in 
$
\on{CAlg}(\on{Fun}(\on{Sp}^{\Delta^{op}}, \Mod_{Hk})).
$
Hence, this is a map of commutative algebra objects. Note that $F \circ \colim_{\Delta^{op}}$ is symmetric monoidal because the smash product in spectra commutes with geometric realizations of simplicial objects. We may now conclude that the map  (\ref{colimitexchangen}) is a morphism of symmetric monoidal functors, and in fact an equivalence. 
\end{proof}

Furthermore, both $\thh(-)$ and $\hh^k(-)$, being symmetric monoidal, send $E_n$-algebras to $E_{n-1}$-algebras. Hence we obtain the following:

\begin{proof}[Proof of Proposition \ref{Enalgebras}]

Let $B^R_{cyc}(-)_{\bullet}$ denote the simplicial object whose realization is $\thh^R(-)$ for a commutative ring spectrum $R$ which is the sphere spectrum when it is omitted. 
If $A$ is an $E_n$-algebra in spectra, then  $B_{cyc}^{R}(A \wedge Hk)_{\bullet}$ will be an $E_{n-1}$-algebra in $\on{Sp}^{\Delta^{op}}$. To see this, recall that an $E_n$ algebra is the same thing as an $E_{1}$-algebra in the symmetric monoidal  $\infty$-category $\on{Alg}_{E_{n-1}}(\on{Sp})$. Hence the cyclic bar construction $B^R_{cyc}(A)_\bullet$ may be viewed as a simplicial object in $E_{n-1}$ algebras; given the pointwise symmetric monoidal structure on $\on{Sp}^{\Delta^{op}}$, this is the same thing as an $E_{n-1}$ algebra in this category. We now apply the previous lemma with $F = - \wedge Hk$ to conclude that 
$$
\rho: |B_{cyc}^{Hk}(A \wedge Hk)_{\bullet}| \xrightarrow{\simeq} |B_{cyc}(A)_{\bullet}| \wedge Hk
$$
is a natural equivalence of $E_{n-1}$-algebras. 
\end{proof}
Applying the above discussion to our current setting, we have an equivalence of $E_1$-algebras.
\[Hk\wdg \thh(\sph [\Omega S^3]) \simeq \hh^k(Hk \wdg \sph[\Omega S^3])\]
Arguing as before, we obtain that the second page of the spectral sequence calculating these Hochschild homology groups is given by 
\[E^2 \cong k[x_2] \otimes \Lambda(\sigma x_2).\]
The differentials are again trivial by degree reasons. We obtain that 
\begin{equation}\label{eq equation above 2}
Hk_* B \cong \hh^k_*(Hk \wdg \sph[\Omega S^3]) \cong k[x_2] \otimes \Lambda(y_3).
\end{equation}
We claim that this is a ring isomorphism. By Theorem \ref{thm e2 dgas with polynomial homology}, we know that $Hk \wdg \sph[\Omega S^3]$ is equivalent to the $E_2$ $Hk$-algebra corresponding to the formal $E_\infty$ $k$-DGA with homology $k[x_2]$. Therefore we can consider $Hk \wdg \sph[\Omega S^3]$ as an $E_\infty$ $Hk$-algebra. 

This equips the spectral sequence above with a multiplicative structure that gives the multiplicative structure on the target. The only multiplicative extension problem is resolved once we show that $y_3^2=0$. This follows by the fact that the target is graded commutative. We obtain the following lemma.

\begin{lem} \label{lem k homology of bm}
The $Hk$ homology of $B(m)$ is given by $k$ concentrated in degrees $2m$ and $2m+1$.
\end{lem}
\subsection{Frobenius morphism on the weight decomposition}
In this section, we prove key lemmas regarding the Frobenius map on $B$. Here, $B$ denotes $\thh(\sph[\Omega S^3])$ as before. By the formalism of Section \ref{gradedthh}, the Frobenius structure map 
$$
\varphi: B \to B^{t C_p}
$$
decomposes, for every $m$, into maps
$$
\varphi_m : B(m) \to  B(pm)^{tC_p}.
$$

\begin{lem}\label{lem frobenius are pcompletions}
The maps $\varphi_m$ above are $p$-completions for every $0\leq m$. Furthermore, $B(m)^{tC_p}$ is contractible as a spectrum whenever $p \nmid m$.
\end{lem}
\begin{proof}
First, we need to establish the behaviour of the weight decomposition with respect to the Tate construction. For this, we claim that the coproduct defining the weight splitting is indeed a product, i.e.\ the canonical map 
\[B \simeq \bigvee_{0 \leq m} B(m) \to \prod_{0 \leq m}B(m)\]
is an equivalence. Each $B(m)$ is connective by construction. This, together with Lemma \ref{lem Z homology of Bm} shows that each $B(m)$ is indeed $2m-1$ connected. Therefore at each homotopy group, there is contribution from finitely factors for the product and finitely many cofactors for the coproduct. This shows that the map above is a homotopy isomorphism.

The homotopy orbits functor preserves connectivity and coproducts and therefore also commutes with this product. The fixed points functor is a limit and therefore it also commutes with products. Therefore, 
\[B^{tCp} = \prod_{0\leq m}B(m)^{tCp}.\]
By Lemma \ref{lem frobenius also split}, the Frobenius  $\varphi \co B \to B^{tCp}$ is given by 
\[\prod_{0 \leq m} \varphi_m  \co \prod_{0 \leq m}B(m) \to \prod_{0\leq m}B(m)^{tCp}.\]

First, note that $B(m)^{tC_p}$ is $p$-complete due to  \cite[\RomanNumeralCaps{1}.2.9]{nikolausscholze2018topologicalcyclic}. Therefore, it is sufficient to show that each map $\varphi_m$ is a $\sph/p$-local equivalence where $\sph/p$ denotes the Moore spectrum of $\z/p$.  Since $B = \Sigma_+^\infty \mathcal{L}S^3$, it follows from  Theorem \RomanNumeralCaps{4}.3.7 of  \cite{nikolausscholze2018topologicalcyclic} that the Frobenius $\varphi$ on $B$ is a $p$-completion. Therefore, it is a homotopy isomorphism after applying the functor $\sph/p \wdg -$. Since $\sph/p$ is dualizable, this functor preserves all homotopy limits as well as all homotopy colimits. We deduce that the map
\[\prod_{0 \leq m} \pi_*(\sph/p \wdg  \varphi_m ) \co \prod_{0 \leq m} \pi_*(\sph/p \wdg B(m)) \to \prod_{0\leq m} \pi_*(\sph/p \wdg B(m)^{tCp})\]
is an isomorphism. Therefore, it is an isomorphism at each component and we have that $B(m)^{tC_p}$ is $\sph/p$-locally contractible whenever $p \nmid m$. Since $B(m)^{tC_p}$ is already $\sph/p$-local, i.e.\ $p$-complete, we  deduce that  $B(m)^{tC_p}$ is contractible as a spectrum whenever $p \nmid m$.
\end{proof}

\begin{lem}\label{lem bm is a finite cp space}
As a spectrum with $C_p$-action, $ B(m)$  is the suspension spectrum of a finite $C_p$-CW space.  
\end{lem}
\begin{proof}
There is an alternative construction of $ B(m)$ as the infinite suspension of a $\T$-pointed space. Let $U$ denote the free monoid on $S^2$ in pointed spaces. This is given by 
\[U = \bigvee_{0\leq l}S^{2l}\]
where $S^2$ denotes the standard 2-sphere with a base point and $S^{2l}$ denotes the $l$-fold smash power of $S^2$ with the zero smash power given by $S^0$. As before, the multiplication is given by the concatenation of the smash factors. 

Let $b_\bullet$ denote the cyclic bar construction for $U$. In particular, $b_\bullet$ is a cyclic pointed space given by $U^{\wdg (s+1)}$ in cyclic degree $s$. As before, we can decompose $b_\bullet$ into pieces consisting of $m$-fold products of $S^2$ and obtain a splitting of cyclic pointed spaces. 
\[b_\bullet \cong \bigvee_{0\leq m}b_\bullet(m)\]
Due to Proposition \ref{prop compatibility between realization in spectra and pointed simplicial sets}, Lemma \ref{lem strictification from model cat to inf cat of top} and Lemma \ref{lem cyclic ptd spaces to cyclic spectra}, there is an equivalence of $\T$-spectra
\[ B(m) \simeq \Sigma^{\infty}\lv b_\bullet(m)\rv.\]
Therefore, it is sufficient to show that $\lv b_\bullet(m)\rv$ is equivalent as a pointed space with $C_p$-action to a finite $C_p$-CW complex. 
 
To understand the $C_p$-action on $\lv b_\bullet(m)\rv $, one considers the $p$-subdivision of $b_\bullet(m)$ as a simplicial pointed space denoted by $sd_pb(m)$. This is described for cyclic sets and cyclic spaces in \cite[Section 1]{bokstedt1993cyclotomictraceandalgebraicktheoryofspaces}. Let $\Delta$ denote the simplex category. There is a functor 
\[\Delta \to \Delta\]
given by $[m-1] \to [mp-1]$ on objects and carries a  morphism $f$ to $f \coprod  \cdots \coprod f$. The $p$-subdivision of $b_\bullet(m)$ is given by precomposing with this functor. For instance, we have $sd_pb(m)_{s-1} = b_{ps-1}(m)$. 

There is a canonical $C_p$-action on $sd_pb(m)_{s-1}$ given by cyclic permutations of blocks of $s$-fold smash factors. This action commutes with the face and degeneracy maps and gives a $C_p$-action on the simplicial space $sd_p b(m)$. Furthermore, the realization of $sd_p b(m)$ agrees with the realization of $b_\bullet(m)$ in a way that is compatible with the $C_p$-action \cite[1.11]{bokstedt1993cyclotomictraceandalgebraicktheoryofspaces}. Therefore, it is sufficient to show that $\lv sd_pb(m)\rv$ is a finite pointed $C_p$-CW complex. 

To achieve this, we consider a bisimplicial  set  $b_\bullet'(m)$. This is obtained in the same way we obtain $b_\bullet(m)$ from $U$ except that we use the standard pointed simplicial set $S^2$ instead of the pointed space $S^2$. For instance, if we apply geometric realization to $b_\bullet'(m)$ at cyclic degree $s$, we obtain $b_s(m)$; this follows by the fact that the geometric realization functor preserves smash products. Similarly, we have a bisimplicial set $sd_pb'(m)$ with a $C_p$-action where the geometric realization of $sd_pb'(m)_s$ is  $sd_pb(m)_s$. Furthermore, these   equivalences preserve the $C_p$-action and therefore there is an equivalence of spaces with $C_p$-action
\[\lv sd_pb(m) \rv \simeq \lv sd_pb'(m) \rv.\]
The geometric realization on the right hand side is obtained by first applying geometric realization at each cyclic degree and then applying the geometric realization functor from simplicial spaces to spaces. Since this is naturally equivalent to applying the realization functor from bisimplicial pointed sets to simplicial pointed sets  through the cyclic direction and then applying the geometric realization functor, we can instead apply this procedure and obtain the same $C_p$-space. Therefore, it is sufficient to show that after applying the realization functor to the bisimplicial set $sd_pb'(m)$ in the cyclic direction, the $C_p$-simplicial set we obtain has finitely many non-degenerate simplices. 

By inspection, one sees that every simplex in the pointed simplicial set $sd_pb'(m)_k$ for $k>m$ is in the image of some degeneracy map from $sd_pb'(m)_m$. Therefore in the realization, one only needs to consider finitely many cyclic degrees. Indeed, the realization in this case is a coequalizer of a coproduct of finitely many pointed  simplicial sets with finitely many non-degenerate simplices and therefore results in a pointed simplicial set with finitely many non-degenerate simplices. 
\end{proof}

\begin{rem}
The above argument takes place using point-set level arguments. We refer the reader to the appendix for the compatbility between these constructions and their higher categorical counterparts. 
\end{rem}
\noindent
Proceeding forward, we obtain the following analogue of Lemma 8 in \cite{speirs2020truncatedpolynomial}.
\begin{lem} \label{lem part of frobenius is an equivalence}
The map 
\[\thh(k)^{tCp}\wedge B(m) \to (\thh(k) \wedge B(pm))^{tCp}\] 
is an equivalence.  
\end{lem}
\begin{proof}
This map factors as 
\[\thh(k)^{tCp}\wedge B(m) \xrightarrow{id \wdg \varphi_m} \thh(k)^{tC_p} \wedge B(pm)^{tC_p} \to (\thh(k) \wedge B(pm))^{tCp}\] 
where the second map is the lax monoidal structure map of the functor $(-)^{tC_p}$. This map is an equivalence due to Lemma \ref{lem bm is a finite cp space} and \cite[Lemma 15]{speirs2019ktheorycoordinate}. For the first map, note that $\sph/p \wdg \varphi_m$ is an equivalence due to Lemma \ref{lem frobenius are pcompletions}. Furthermore, we have $\thh(k)^{tC_p} = \sph/p \wdg \thh(k)^{t\T}$, see \cite[proof of \RomanNumeralCaps{4}.4.13]{nikolausscholze2018topologicalcyclic}, and therefore the first map is also an equivalence.
\end{proof}
The Frobenius 
\[\phi_m \co \thh(k)\wedge B(m)\to(\thh(k) \wedge B(pm))^{tCp} \]
factors as
\begin{equation} \label{eq composite defining the frobenius}
    \thh(k)\wedge B(m) \to \thh(k)^{tC_p} \wdg B(m)  \to (\thh(k) \wedge B(pm))^{tCp}
\end{equation}
where the first map is induced by the Frobenius on $\thh(k)$ and the second map is the equivalence given in Lemma \ref{lem part of frobenius is an equivalence}.  
\begin{cor}\label{cor frobenius is an isomorphism fur suff large}
The Frobenius $\phi_m$ described above induces an isomorphism in homotopy groups for sufficiently large degrees. In particular, $\pi_*\phi_m$ is an isomorphism in degree $2m+1$ and above. Furthermore, $\pi_*\phi_m^{h\T}$ is also an isomorphism in degree $2m+1$ and above.
\end{cor}
\begin{proof}
It is sufficient to show that the first map of the composite in \eqref{eq composite defining the frobenius} is an equivalence in sufficiently large degrees. 

Let $E$ denote the fiber of the Frobenius map of $\thh(k)$. We need to show that $E \wdg B(m)$ is bounded from above. 

The Frobenius on $\thh(k)$ is an equivalence on connective covers due to Proposition 6.2 in \cite{bhatt2019thhandintegralpadichodge}. Since $\thh(k)$ is connective, this shows that $\tau_{\leq -2}E \simeq E$.

Furthermore, Lemma \ref{lem Z homology of Bm} implies that $B(m)$ is a finite spectrum, see  \cite[\RomanNumeralCaps{2}.7.4]{schwede-book}. This is also implied by Lemma \ref{lem bm is a finite cp space}. This shows that $E \wdg B(m)$ is built from $E$ by taking suspensions, completing triangles and taking retracts finitely many times. This implies that $E \wdg B(m)$ is also bounded from above.

To be precise, one can show using Lemma \ref{lem Z homology of Bm} that there is a fiber sequence 
\[
\Sigma^{2m} \sph \to B(m) \to \Sigma^{2m+1} \sph 
\]
which in turn gives a fiber sequence 
\[\Sigma^{2m} E \to E \wdg B(m) \to \Sigma^{2m+1} E. \]
This shows that $E\wdg B(m)$ is trivial in homotopy groups in degree $2m$ and above. Therefore, $\pi_*\phi_m$ is an isomorphism in degree $2m+1$ and above. 

To show that $\pi_*\phi_m^{h\T}$ is also an isomorphism in degree $2m+1$ and above, it is sufficient to show that $(E\wdg B(m))^{h\T}$ is bounded above degree $2m$. This follows by the fact that the homotopy fixed points spectral sequence described in Section \ref{sec calculations of the fixed points and tate } preserves coconnectivity.
\end{proof}
\section{Calculation of $\tc_*(\thh(\fp))$}\label{sec calculations of the fixed points and tate }
In this section,  we calculate $\tc_*(\thh(k))$. This provides the relative $K$-theory groups for the map 
\[\thh(k)\to Hk.\]
Our methods closely follow those of Speirs in \cite[Sections 5 and 6]{speirs2020truncatedpolynomial}.

First, we show that our weight splittings also result in splittings at the level of negative cyclic homology and periodic cyclic homology.

Let 
\[\thh(X)(m) \textup{, } \tc^-(X)(m)\textup{ and }\tp(X)(m) \]
denote 
\[\thh(k) \wdg B(m) \textup{, } (\thh(k) \wdg B(m))^{h\T}\textup{\ and }(\thh(k) \wdg B(m))^{t\T}\] 
respectively. 

\begin{prop}\label{prop negative tc and tp split}
There are equivalences 
\[\tc^-(X) \simeq \prod_{0 \leq m} \tc^-(X)(m) \textup{\ and }\tp(X) \simeq \prod_{0 \leq m} \tp(X)(m).\]
\end{prop}
\begin{proof}
This is similar to the argument at the beginning of the proof of Lemma \ref{lem frobenius are pcompletions}. Due to connectivity reasons, the coproduct in \eqref{eq thhx splits as a sum} is at the same time a product. Therefore it commutes with taking fixed points which is a homotopy limit. This gives the first splitting. 

The coproduct in \eqref{eq thhx splits as a sum} commutes with taking homotopy orbits because homotopy orbits is a homotopy colimit. Furthermore, the homotopy orbits functor preserves connectivity and therefore the coproduct splitting one obtains after taking homotopy orbits is at the same time a product. It is also clear that the canonical map from the homotopy orbits to homotopy fixed points respects this splitting.  Therefore, we obtain the product splitting for the Tate construction too. 
\end{proof}

To calculate $\tc^-(X)(m)$ and $\tp(X)(m)$, we use the homotopy fixed point and the Tate spectral sequences. 

For a $\T$-spectrum $E$, the second page of the homotopy fixed point spectral sequence is given by 
\[E^2 = \Z[t] \otimes \pi_{*}E \Longrightarrow \pi_{*}E^{h\T}\]
where $\text{deg}(t) = (-2,0)$. This is a second quadrant spectral sequence.

The Tate spectral sequence is given by 
\[E^2 = \Z[t,t^{-1}] \otimes \pi_{*}E \Longrightarrow \pi_{*}E^{t\T}\]
where $\text{deg}(t) = (-2,0)$ as before. This is a conditionally convergent first half plane spectral sequence. Furthermore, this is a  multiplicative spectral sequence when $E$ is a $\T$-ring spectrum \cite[4.3.5]{hesselholt2003ktheoryoflocalfields}.

There is a subtlety regarding the way we incorporate the multiplicative structure of the Tate spectral sequence into our calculations. We use the Tate spectral sequence  to calculate the homotopy groups of 
\[(\thh(k) \wdg B(m))^{t\T}.\]
However, $B(m)$ is not a ring spectrum for $m>0$; it does not contain a unit and it is not closed under the multiplication in $B$. 

Since $\thh(-)$ is a symmetric monoidal functor from $E_1$-ring spectra to cyclotomic spectra \cite[Section \RomanNumeralCaps{4}.2]{nikolausscholze2018topologicalcyclic}, we obtain a splitting of $\T$-ring spectra by using Theorem \ref{thm thhfp as an e2 algebra}. 
 \[\thh(X) \simeq \thh(k) \wdg \thh(\sph[\Omega S^3]) := \thh(k) \wdg B.\]
 
 Therefore, the Tate spectral sequence for the right hand side of this equality is multiplicative. Indeed, this is why we needed Theorem \ref{thm thhfp as an e2 algebra} for our calculations. 
 
 The weight splitting of $\thh(k) \wdg B$ also splits the Tate spectral sequence for  $\thh(k) \wdg B$ through the Tate spectral sequences calculating $\thh(k) \wdg B(m)$ for all $m$. This can be seen by noting that each  $\thh(k) \wdg B(m)$ is actually a retract of $\thh(k) \wdg B$ as a $\T$-spectrum. Using this, we incorporate the multiplicative structure on the Tate spectral sequence for  $\thh(k) \wdg B$ into the Tate spectral sequence calculating $\thh(k) \wdg B(m)$. 
 
 Due to Lemma \ref{lem k homology of bm}, the second page of the Tate spectral sequence for  $\thh(k) \wdg B(m)$ is given by 
 \[E^2 = k[t,t^{-1},x_2] \{y_m,z_m\}\]
 where $\deg(y_m) = (0,2m)$ and $\deg(z_m) = (0,2m+1)$ and $\deg(x_2) = (0,2)$. We start with the following lemma.
 
 \begin{lem}\label{lem trivial differentials on the tate spectral sequence}
 In the Tate spectral sequence for $\thh(k) \wdg B(m)$, $z_m$  is an infinite cycle. In the Tate spectral sequence for $\thh(k) \wdg B$, $x_2$, $t$ and $t^{-1}$ are  infinite cycles.
 \end{lem}
 \begin{proof}
 As described in  \cite{nikolausscholze2018topologicalcyclic} after Corollary \RomanNumeralCaps{4}.4.10, there is a map of $\T$-equivariant ring  spectra $H\Z_p^{triv}\to \thh(\fp)$ where $H\z_p^{triv}$ is given the trivial $\T$-structure. This, together  with the ring map $\fp \to k$ gives a map of $\T$-equivariant  ring spectra
 \[H\Z_p^{triv}\to \thh(k).\]
 The Tate spectral sequence calculating $H\z_p^{triv} \wdg B(m)$ has 
 \[E^2= \Z_p[t,t^{-1}]\{y_m,z_m\}.\]
 In this spectral sequence, $z_m$ is an infinite cycle due to degree reasons. Furthermore, this spectral sequence maps into the Tate spectral sequence for $\thh(k) \wdg B(m)$ in a way that carries $z_m$ to $z_m$ on the $E^2$ page. This shows that $z_m$ is an infinite cycle in the Tate spectral sequence calculating $\thh(k) \wdg B(m)$.
 
 Since $B$ is a cyclotomic ring spectrum, there is a map of cyclotomic spectra 
 \[\sph^{triv}\to B\]
 given by the unit of $B$. This map induces a map of spectral sequences between the Tate spectral sequence for $\thh(k)$ and the Tate spectral sequence for $\thh(k) \wdg B$ that carries $x_2$ to $x_2$, $t$ to $t$ and $t^{-1}$ to $t^{-1}$. In the first spectral sequence, everything is in even degrees and therefore all the  differentials are trivial. This gives the desired result.
 \end{proof}

\begin{prop} \label{prop calculation of tate and homotopy fixed points}
Let $m= p^vm^\prime$ for $p \nmid m^\prime$.  There are isomorphisms
\[\pi_{2r+1}(\thh(k) \wdg B(m))^{t\T} = W_v(k)\]
and 
\begin{equation*}
\pi_{2r+1}(\thh(k) \wdg B(m))^{h\T} = 
\begin{cases}
                                   W_{v+1}(k) & \text{if $m \leq r$} \\
                                   W_{v}(k) & \text{if $r < m$}
  \end{cases}
\end{equation*}
for all $z$. The even homotopy groups in both cases are trivial. Furthermore, the canonical map $can$ is an isomorphism for $r<m$.
\end{prop}
\begin{proof}
This proof is a direct adaptation of \cite[Proposition 12]{speirs2020truncatedpolynomial}. We start with the case $p \nmid m$. In this case, 
\[\thh(k)^{tC_p} \wdg B(m)^{tC_p} \simeq 0\]
due to Lemma \ref{lem frobenius are pcompletions}. The lax monoidal structure map of the functor $(-)^{tC_p}$ is an equivalence in this situation due to Lemma \ref{lem bm is a finite cp space} and \cite[Lemma 15]{speirs2019ktheorycoordinate}. We conclude that 
\[(\thh(k) \wdg B(m))^{tC_p} \simeq 0.\]
Since $\thh(k) \wdg B(m)$ is $p$-complete (for every $m$), $(\thh(k) \wdg B(m))^{t\T}$ is also $p$-complete \cite[Section 2.3]{bhatt2019thhandintegralpadichodge}. Therefore we have 
\[
(\thh(k) \wdg B(m))^{t\T} \simeq ((\thh(k) \wdg B(m))^{tC_p})^{h\T} \simeq 0.
\]
due to Lemma \RomanNumeralCaps{2}.4.2 of \cite{nikolausscholze2018topologicalcyclic}.

Recall that the second page of the Tate spectral sequence for $(\thh(k) \wdg B(m))^{t\T}$ is given by 
\[
 E^2 = k[t,t^{-1},x_2] \{y_m,z_m\}.
\]
Considering Lemma \ref{lem trivial differentials on the tate spectral sequence} together with the multiplicative structure, one sees that the first non-trivial differential on $y_m$ determines the rest of the non-trivial differentials on this spectral sequence. In this case, we have $E^\infty = 0$ and therefore up to a unit, we have
\[d^2y_m = tz_m\]
which is the only differential that guarantees this. See the following picture of the $E^2$-page.

\begin{tikzpicture}
\matrix (m) [matrix of math nodes,
             nodes in empty cells,
             nodes={minimum width=6.95ex,
                    minimum height=7ex,
                    outer sep=-5pt},
             column sep=-0.3ex, row sep=-3.9ex,
             text centered,anchor=center]{
 \vdots  &   \strut  &     &     &  & \vdots &
           & &&\\
2m+3 &   & t^2x_2z_m    & &tx_2z_m & & x_2z_m &  & t^{-1}x_2z_m &\\
2m+2 &   & t^2x_2y_m    & &tx_2y_m & &x_2 y_m & & t^{-1}x_2y_m &&\\
2m+1 & \cdots  & t^2z_m    & &tz_m & & z_m & & t^{-1}z_m &\cdots\\
2m &   & t^2y_m    & &ty_m & & y_m & & t^{-1}y_m &\\
 \vdots  &     &    0&    &  0   &    &0 &&0\\
\quad\strut &   \cdots &  4  &    &  2 &    &  0 &         & -2 &\cdots \\
};
   \draw[-stealth] (m-5-7) -- (m-4-5);
   \draw[-stealth] (m-3-7) -- (m-2-5);
   
      \draw[-stealth] (m-5-5) -- (m-4-3);
   \draw[-stealth] (m-3-5) -- (m-2-3);
   
      \draw[-stealth] (m-5-9) -- (m-4-7);
   \draw[-stealth] (m-3-9) -- (m-2-7);
   
\draw[thick] (m-7-1.north) -- (m-7-10.north east) ;
\end{tikzpicture}

We showed that $\pi_k((\thh(k) \wedge B(m))^{t \T})$ vanishes for each $k$. We now truncate the Tate spectral sequence, removing the first quadrant to obtain the homotopy fixed point spectral sequence. In this case, the classes $x_2^n z_m$  will no longer be hit by any differentials and so survive to the $E^\infty$ page; hence we conclude that $E^\infty= k[x_2]\{z_m\}$, where $z_m$ is of degree $2m +1$.

Now suppose that $m=p^v m'$, where $(p,m')=1$. We use an induction argument on $v$. Suppose then that the claim is true for all integers less than or equal to $v$. Recall, from Corollary \ref{cor frobenius is an isomorphism fur suff large} that the Frobenius 
$$
\phi_m^{h \T}: (\thh(k) \wedge B(p^v m'))^{h\T} \to (\thh(h) \wedge B(p^{v +1} m'))^{t \T}   
$$
will be a $\pi_*$ isomorphism in sufficiently high degrees. Using this fact, we may conclude that  
\begin{equation}\label{eq homotopy of the tate in the inductive argument}
\pi_{2 r +1} (\thh(k) \wedge B(p^{v+1} m'))^{t \T} \cong W_{v +1}(k)
\end{equation}
for sufficiently large $r$ and that the even homotopy groups are trivial in sufficiently large degrees. Since $(-)^{t\T}$ is lax monoidal and $\thh(k)$ is a $\T$-equivariant $H\Z$-algebra  \cite[\RomanNumeralCaps{4}.4.10]{nikolausscholze2018topologicalcyclic}, we deduce that $(\thh(k) \wedge B(p^{v+1} m'))^{t \T}$ is an $H\Z^{t\T}$-module. It follows from the Tate spectral sequence that  
\[\pi_*\Z^{t\T} = \Z[t,t^{-1}]\]
with $\lv t \rv=2$. From this, we deduce that $(\thh(k) \wedge B(p^{v+1} m'))^{t \T}$ is periodic in homotopy and this shows that the equality in \ref{eq homotopy of the tate in the inductive argument} holds for every $r$ and that the even homotopy groups are trivial.

It now remains to compute 

$$
\pi_*(\thh(k) \wedge B(p^{v+1}m'))^{h \T}
$$
First, we observe that in the Tate spectral sequence, 
$$
d^{2v+2}(y_{p^{v+1}m'})= t(x_2t)^v z_{p^{v+1}m'}
$$
at least up to multiplication by a unit. Since this spectral sequence is multiplicative, this is the only differential that guarantees \eqref{eq homotopy of the tate in the inductive argument}.  There are no other non-trivial differentials. 

This completes the calculation of $\pi_{2r+1}(\thh(k) \wedge B(m))^{h \T}$, up to extension problems. Note that we only need to do this for the homotopy fixed point calculation. So we reduce the homotopy fixed point case as in  \cite[Proposition 12]{speirs2019ktheorycoordinate}. This will follow from us showing that $\pi_{2r}(\thh(k) \wedge B(m))^{h\T}$ is cyclic (and so completely determined by its length, which we may extract from the  $E^{\infty}$ page) as a $\pi_0 \thh(k)^{h\T} \cong W(k)$-module. The argument proceeds by exhibiting a $\T$-equivariant map
$$
\thh(k)[2m +1] \to \thh(k) \wedge B(m)
$$
which induces a map of homotopy spectral sequences. Choose an element 
$$\alpha \in \pi_{2r+1}(\thh(k) \wedge B(m))^{h \T},
$$
and let $\overline{\alpha}= t^{a}x^a z_m$ be its image in the $E^\infty$  page. This lies in the image of $t^a x^a$ in the $E^\infty$ page of the spectral sequence computing $\thh(k)^{h \T}$ where the extension problem has been solved; thus $p^a$ lifts $t^a x^a$ up to a unit. Since the map 

$$
\thh(k)[2m +1] \to \thh(k) \wedge B(m)
$$
corresponding to $z_m$ is $W(k)$ linear, one concludes that $\alpha = p^a z_m$. 

Hence we have shown that 
\begin{equation*}
\pi_{2r+1}(\thh(k) \wdg B(m))^{h\T} = 
\begin{cases}
                                   W_{v+1}(k) & \text{if $m \leq r$} \\
                                   W_{v}(k) & \text{if $r < m$}
  \end{cases}
\end{equation*}
It only remains to verify that 
$$
can_*: \pi_{2r+1} (\thh(k) \wdg B(m))^{h \T} \to \pi_{2r +1}(\thh(k) \wdg B(m))^{t \T} 
$$
is an isomorphism for $r< m$. This follows from the fact that $can$ induces a morphism of spectral sequences from the homotopy fixed point to the tate spectral sequence; in these degrees, the induced maps on the $E^\infty$ page are equivalences as there will be no contribution there from the $1$st quadrant terms. 
\end{proof}

The relative topological cyclic homology  $\tc(X,k)$ is the fiber of the map 
\[\tc^-(X,k)\xrightarrow{\phi - can} \tp(X,k).\]
With the splitting in Proposition \ref{prop negative tc and tp split}, this is the fiber of the following map.

\[
\prod_{\substack{1\leq m^\prime \\ p \nmid m^\prime}} \prod_{0 \leq v} \tc^-(X)(p^vm^\prime) \xrightarrow{\phi - can} \prod_{\substack{1\leq m^\prime \\ p \nmid m^\prime}} \prod_{0 \leq v} \tp(X)(p^vm^\prime)
\]
Recall that $\phi$ increases the weight degree by a factor of $p$, see Lemma \ref{lem frobenius also split}. 

We claim that, $\phi-can$ is surjective in homotopy groups. To see this, fix an $m^\prime$ with $p\nmid m^\prime$ and consider the degree $2r+1$ homotopy groups of the factors corresponding to this $m^\prime$. If $r< m^\prime$, then $can$ is an isomorphism due to Proposition \ref{prop calculation of tate and homotopy fixed points}. This shows that $\phi-can$ is surjective in homotopy since $\phi$ increases the weight degree by a factor of $p$. If $m^\prime \leq r$, then the surjectivity for $p^vm^\prime \leq r$ follows by the fact that $\phi$  is an isomorphism in these cases and that $\tp(m') \simeq 0$, see Corollary \ref{cor frobenius is an isomorphism fur suff large}. For $v$ with $p^v m^\prime >r$, surjectivity is observed by noting that $can$ is an isomorphism in these cases.

 All the non-trivial homotopy groups are in odd degrees due to Proposition \ref{prop calculation of tate and homotopy fixed points}. This, together with the surjectivity of $\pi_*(\phi-can)$ show that for every positive $m'$ with $p\nmid m'$, there is a short exact sequence 

\[0 \to \tc_*(X)(m') \to \prod_{0 \leq v }\tc^-_*(X)(p^vm^\prime)\to \prod_{0 \leq v }\tp_*(X)(p^vm^\prime) \to 0\]
where 
\[\tc_*(X,k) = \prod_{\substack{1\leq m^\prime \\ p \nmid m^\prime}}\tc_*(X)(m')\]
with this notation.

Restricting to degree $2r+1$, we obtain the following diagram where the horizontal lines are short exact sequences. 
\begin{equation*}
    \begin{tikzcd}[row sep=normal, column sep = small]
    0 \ar[r]&\ar[d,"\phi -can"]\prod_{s \leq v}W_v(k) \ar[r]& \ar[d,"\phi-can"] \prod_{0 \leq v }\tc^-_{2r+1}(X)(p^vm^\prime) \ar[r]&\ar[d,"\overline{\phi-can}"] \prod_{0 \leq v < s} W_{v+1}(k) \ar[r] &0 \\
        0 \ar[r]&\prod_{s \leq v}W_v(k) \ar[r]& \prod_{0 \leq v }\tp_{2r+1}(X)(p^vm^\prime) \ar[r] &\prod_{0 \leq v < s} W_{v}(k) \ar[r] &0 
    \end{tikzcd}
\end{equation*}
Here, $s$ is the smallest non-negative integer $v$ such that $r<p^vm'$. The vertical map on the left hand side is an isomorphism due to Proposition \ref{prop calculation of tate and homotopy fixed points} and the fact that $\phi$ increases the weight degree. Therefore, the left hand side does not contribute to the cyclic homology groups. 

Now we consider the right hand vertical map. Due to Corollary \ref{cor frobenius is an isomorphism fur suff large}, $\phi$ is an isomorphism on  $W_{v+1}(k)$ for each $v$ except for $v = s-1$ where it is trivial because its image lies in the first part of the short exact sequence. The kernel of the right hand side is given by $W_s(k)$. Indeed, the kernel  is the image of an injective map 
\[f \co W_s(k) \to \prod_{0 \leq v < s} W_{v+1}(k)\]
given by the identity map for $v= s-1$. The rest of this map is defined via a downward induction on $v$. Given a map to $W_{v+1}$, the map to $W_{v}$ is given by 
\[W_{s}(k)\to W_{v+1}(k) \xrightarrow{can} W_{v}(k) \xrightarrow{\phi^{-1}}W_v(k)\]
where the first map is the given map. In conclusion, the contribution to $\tc_{2r+1}(X)$ from $m'$ is given by $W_s(k)$.

We deduce that 
\[\tc_{2r+1}(X,k) = \prod_{\substack{1\leq m^\prime \leq r \\ p \nmid m^\prime}} W_s(k) \textup{\ and \ } \tc_{2r}(X,k)=0 \]
for every integer $r$. It follows from Proposition 1.10 and Example 1.11 of  \cite{hesselholt2015thebigderhamwittcomplex}
the the right hand side is indeed the big Witt vectors  $\mathbb{W}_r(k)$ of length $r$. This finishes the proof of Theorem \ref{thm algebraic k theory of thh k}.
\section{Algebraic $K$-theory of $\tc(\f_p)$} \label{sec algebraic k theory of tc fp}
 
In this section, we show that there is an equivalence of spectra
\[K(\tc(\fp)) \simeq K(\z_p).\] 
We start by proving the following classification result. 
 \begin{prop} \label{prop unique dga with exterior homology}
 Let $R$ be a commutative ring. There is a unique $R$-DGA with homology ring $\Lambda(x_{-1})$, i.e.\ the exterior algebra over $R$ with a single generator in degree $-1$.
 \end{prop}
 
 The first author and P\'eroux prove this result when $R$ is a field \cite{bayindirperoux}. We describe how their proof generalize to commutative rings. 
 
 Let $X$ be the $HR$-algebra corresponding to the formal $R$-DGA with homology ring $\Lambda(x_{-1})$ and let $E$ be the $HR$-algebra corresponding to another $R$-DGA with homology ring $\Lambda(x_{-1})$. We need to show that $E$ is weakly equivalent to $X$ as an $HR$-algebra. 
 
 For this, we use Hopkins--Miller obstruction theory which provides obstructions to lifting a map of monoids in the homotopy category of $HR$-modules to a map of $HR$-algebras \cite{rezk1998notes}. 
 
 Since homotopy groups of $E$ and $X$ are free as graded $R$-modules, maps between $E$ and $X$ in the homotopy category of $HR$-modules is simply given by maps of homotopy groups that are maps of graded $R$-modules, see \cite[\RomanNumeralCaps{4}.4.1]{elmendorf2007rings}. Again by \cite[\RomanNumeralCaps{4}.4.1]{elmendorf2007rings}, smash powers of $X$ and $E$ have homotopy given by tensor products in graded $R$-modules. From these, we conclude that $X$ and $E$ are isomorphic as monoids in the homotopy category of $HR$-modules. 
 
Again because $\pi_*X$ and $\pi_*E$ are free as $R$-modules, we can apply Hopkins-Miller obstruction theory to lift this isomorphism to a weak equivalence of $HR$-algebras. The obstructions to this lift lies in the Andr\'e--Quillen cohomology groups 
\[\textup{Der}^{s+1}(\Lambda(x_{-1}), \Omega^{s}\Lambda(x_{-1})) \textup{\ for \ } s\geq 1.\]
It follows by \cite[3.6]{quillen1968cohomologyofcommutativerings} that these cohomology groups are equivalent to the Hochschild cohomology groups 
\[
\textup{Ext}_{\Lambda(x_{-1}) \otimes \Lambda(x_{-1})^{op}}^{s+2}(\Lambda(x_{-1}), \Omega^{s}\Lambda(x_{-1})) \textup{\ for \ } s\geq 1.
\]

There is an automorphism of $\Lambda(x_{-1}) \otimes \Lambda(x_{-1})^{op}$ given by 
\[x_{-1} \otimes 1 \to x_{-1} \otimes 1 \textup{\ and \ } 1 \otimes x_{-1} \to x_{-1} \otimes 1 - 1 \otimes x_{-1}.\]
Precomposing with this automorphism makes the action of the second factor of  $\Lambda(x_{-1}) \otimes \Lambda(x_{-1})^{op}$  on $\Lambda(x_{-1})$ trivial  and the action of the first factor the canonical nontrivial one. Via base change with respect to the inclusion of the second factor of $\Lambda(x_{-1}) \otimes \Lambda(x_{-1})^{op}$, one sees that these obstructions lie in the groups
\[
\textup{Ext}_{ \Lambda(x_{-1})}^{s+2}(R, \Omega^{s}(R \oplus \Omega R)) \textup{\ for $s \geq 1$.}\]
Using the standard resolution of $R$ as a  $\Lambda(x_{-1})$-module it is clear that these groups are trivial due to degree reasons. 

This shows in particular that $\tc(\fp)$ and $C^*(S^1,\z_p)$ are both formal as $E_1$ $H\z_p$-algebras and therefore quasi-isomorphic to each other. It follows from the theorem of the heart in \cite[4.8]{antieau2019ktheoreticobstructionstobddtstructures} that the  algebraic $K$-theory of $C^*(S^1,\Z_p)$ and therefore $\tc(\fp)$  is $\kth(\z_p)$.
\begin{rem}
 For a perfect field $k$ of characteristic $p$, Hesselholt and Madsen identify the topological cyclic homology of $k$ as  
 \[\tc(k) \simeq H\z_p \vee \Sigma^{-1}H(\textup{coker}(F-1))\]
 where $F$ denotes the Frobenius map on the ring of Witt vectors of $k$ \cite[Theorem B]{hesselholtwittvectors1997k}. If $k$ is finite, then  $\textup{coker}(F-1) = \z_p$. Therefore, Proposition \ref{prop unique dga with exterior homology} applies to show that $\tc(k) \simeq  C^*(S^1,\Z_p)$. This, together with \cite[4.8]{antieau2019ktheoreticobstructionstobddtstructures}, implies that 
 \[\kth(\tc(k)) \simeq \kth(\z_p)\]
 for every finite field $k$.
 If $k$ is an algebraically closed field of characteristic $p$, then $\textup{coker}(F-1) =0$ and again, this identifies  $\kth(\tc(k))$ with $ \kth(\z_p)$.
\end{rem}
 
 \section{Appendix A. Rectification of cyclic and graded objects}
In this short appendix, we prove rectification results regarding cyclic objects in various $\infty$-categories which we use in the main part of the text. 

We first recall the following compatibility between realization of a cyclic object using point-set models and in the $\infty$-categorical setting. In what follows, $N(-)$ denotes the nerve functor,  $N\mathbb{T}-Top$ denotes the nerve of the category of $\mathbb{T}$-equivariant topological spaces, and the vertical maps are the Dwyer-Kan localization maps.      
\begin{prop}[Nikolaus-Scholze] \label{prop rectification cyclic spaces to cyclic spectra}
Let $Top$ denote the category of  compactly generated weak Hausdorff spaces. The diagram
\begin{equation} 
 \begin{tikzcd}
 N(\on{Fun}(\Lambda^{op}, Top))_{prop} \ar[r, "|-| "] \ar[d]& N \T-Top \ar[d]\\
 \on{Fun}( \Lambda^{op}, \mathcal{S}) \ar[r,"\colim"]& \mathcal{S}^{B \T}
 \end{tikzcd}
 \end{equation}
 commutes.\footnote{The $(-)_{prop}$ subscript means we restrict to cyclic spaces whose underlying simplicial space is \emph{proper}. A simplicial space is proper if for each $n$, the inclusion $\cup_{i=0}^{n-1}s_i(X_{n-1}) \hookrightarrow X_n $ is a Hurewicz cofibration.} 
\end{prop}

\begin{proof}
This is Corollary B.14 in \cite{nikolausscholze2018topologicalcyclic}. 
\end{proof}

\noindent We shall need follow pointed analogue of this result.  

\begin{lem}\label{lem cyclic ptd spaces to cyclic spectra}
Let $Top_*$ denote the category of pointed compactly generated weak Hausdorff spaces. The diagram
\begin{equation} 
 \begin{tikzcd}
 N(\on{Fun}(\Lambda^{op}, Top_*))_{prop} \ar[r,"|-|"] \ar[d]& N \T-Top_* \ar[d]\\
 \on{Fun}( \Lambda^{op}, \mathcal{S}_*) \ar[r,"\colim"]& \mathcal{S}_*^{B \T}
 \end{tikzcd}
 \end{equation}
 commutes. 
\end{lem}

\begin{proof}
This boils down to the analogous statement for unpointed spaces, together with the fact that the geometric realization of a simplicial object in pointed spaces is canonically pointed, and that the geometric realization of a pointed (para)cyclic space is the geometric realization of the underlying simplicial  pointed space. 
\end{proof}

As we occasionally work with point-set models in the setting of graded topological spaces, we need to show  that various point-set constructions are compatible with constructions at the level of $\infty$-categories. For this, let 
$$
Top_*^{\Z} := \on{Fun}(\Z, Top_*)
$$
denote the category of functors  $F: \Z \to Top_*$.  With the Day convolution  product (cf. \cite{day1970closed}) and the projective model structure (see eg. \cite[Theorem 4.1]{batanin2013homotopy}), this becomes a symmetric monoidal model category. In our constructions, we sometimes start with a monoid in this category and use the following compatibility result.

\begin{lem}\label{lem strictification from model cat to inf cat of top}
Let $X^{str}$ be a cofibrant monoid in $Top_*^{\Z}$ and let $b_\bullet^{str}$ denote the corresponding $\Lambda$-object in $Top_*^{\Z}$ given by the cyclic bar construction on $X^{str}$. We denote by $X$ the image of $X^{str}$ under the canonical map 
\[
N((\on{Alg}(Top_*^{\Z})^c)[W^{-1}] \to \on{Alg}(N((Top_*^{\Z})^c)[W^{-1}]).\footnote{On the left, $W$ denotes the class of algebra morphisms $f:A \to B$ which are weak equivalences when regarded as morphisms on  $Top_*^{\Z}$. On the right $W$ denotes the class of weak equivalences in   $Top_*^{\Z}$ itself. These are the object-wise weak equivalences for this functor category.}
\]
Furthermore, the image of $b_\bullet^{str}$ under the map 
$$N((Top_*^{\Z \times \Lambda^{op} })^c)[W^{-1}] \to N((Top_*)^c)[W^{-1}]^{\Z \times \Lambda^{op}} \simeq \mathcal{S}_*^{\Z \times \Lambda^{op}} 
$$
\noindent
is denoted by $b_\bullet$. There is an equivalence 
\[b_\bullet \simeq b(X)\]
where $b(X)$ is the cyclic bar construction on $X$. 
\end{lem}

\begin{proof}
This amounts to showing that the following diagram commutes
\begin{equation} 
 \begin{tikzcd}
 N(\on{Alg}(Top^\Z_*))[W^{-1}] \ar[r,"\on{\simeq}"] \ar[d, ""]& \on{Alg}(\mathcal{S}^\Z_*) \ar[d, ""]\\
N(Top_*^{\Z \times \Lambda^{op}})[W^{-1}] \ar[r,"\simeq"]& \mathcal{S}_*^{\Z \times \Lambda^{op}}
 \end{tikzcd}
 \end{equation}
where the vertical arrows represent the strict version of the cyclic bar construction and the $\infty$-categorical version, respectively. The cyclic object $b_\bullet$  as an object in the bottom right hand corner, obtained by applying the cyclic bar construction to $X^{str}$, is given by 
$$
\Lambda^{op} \to \assact \xrightarrow{X^{str}} N(Top_*^\Z)^{\wedge}_{\on{act}} \xrightarrow{\wedge} N(Top_*^\Z);
$$
this follows since $\on{Alg}(N(Top_*^{\Z}))[W^{-1}]) \simeq N(\on{Alg}(Top_*^{\Z}))[W^{-1}]$  (see e.g.   \cite[Theorem 4.1.8.4]{lurie2016higher};  the category $Top_*^\Z$ doesn't satisfy these properties on the nose, but may be replaced by a Quillen equivalent one which does).

In the above, we are implicitly using  \cite[Corollary 4.2.4.7]{lurie2009higher}
which implies that the diagram $N(\Lambda^{op}) \to N(Top_*^{\Z})[W^{-1}] \simeq \mathcal{S}_*^\Z$ may be straightened, in that it arises from the nerve construction applied to a functor $\Lambda \to Top_*^\Z$, and it determines this functor uniquely.

In cyclic degree $n$, this cyclic object is given by the $n+1$ pointed smash product
$$
X^{str} \wedge...\wedge X^{str}
$$
with the standard Hochschild structure maps, encoded by 
the map $\Lambda^{op} \to \assact $.
This is none other than the cyclic object
$$
\Lambda^{op} \to \assact  \xrightarrow{X} {\mathcal{S}_*}^{\wedge}_{\on{act}} \xrightarrow[]{\wedge} \mathcal{S}_*
$$
which is the bar construction $b(X)$ on $X$. 
\end{proof}

\providecommand{\bysame}{\leavevmode\hbox to3em{\hrulefill}\thinspace}
\providecommand{\MR}{\relax\ifhmode\unskip\space\fi MR }
\providecommand{\MRhref}[2]{%
  \href{http://www.ams.org/mathscinet-getitem?mr=#1}{#2}
}
\providecommand{\href}[2]{#2}

\end{document}